\documentclass[a4paper,11pt,oneside,reqno]{amsart}
\usepackage[ansinew]{inputenc}
\usepackage{xcolor}
\usepackage{amsfonts}
\usepackage{amssymb}
\usepackage{amsmath,amsthm,amsfonts,amssymb}
\usepackage[left=3cm,right=3cm,top=4cm,bottom=4cm]{geometry}
\usepackage[dvips]{graphicx}
\usepackage[english]{babel}

\theoremstyle{plain}
\newtheorem{lem}{Lemma}[section]
\newtheorem{thm}[lem]{Theorem}

\theoremstyle{definition}
\newtheorem{Def}[lem]{Definition}

\newcommand{\bn}{\mathbf{n}}
\newcommand{\hp}{\phi_i}
\newcommand{\Sp}{\psi_i}
\newcommand{\fd}{\mathcal{L}(\bu;\bb)}
\newcommand{\bcu}{\mathbf{U}}
\newcommand{\bcm}{\mathbf{M}}
\newcommand{\bu}{\mathbf{u}}
\newcommand{\bv}{\mathbf{v}}
\newcommand{\bw}{\mathbf{w}}

\newcommand{\bb}{\mathbf{b}}
\newcommand{\bm}{\mathbf{m}}
\newcommand{\bc}{\mathbf{c}}
\newcommand{\gu}{\Gamma(\bu)}
\newcommand{\gut}{\Gamma(\bu(t))}

\newcommand{\mce}{\mathcal{E}}
\newcommand{\ms}{\mathbb{S}}
\newcommand{\but}{\bu_t}

\newcommand{\A}{\mathcal{A}}

\DeclareMathOperator{\tr}{tr}

\newcommand{\ve}{\mathbb{V}}

\newcommand{\h}{\mathbb{H}}
\newcommand{\bus}{\bu(s)}
\newcommand{\bbs}{\bb(s)}
\newcommand{\eps}{\varepsilon}

\DeclareMathOperator{\curl}{curl}
\DeclareMathOperator{\Div}{div}
\DeclareMathOperator{\Span}{Span}
\title[Qualitative properties of the solution to MHD equations for bipolar fluids]{Some qualitative properties of the solution to the Magnetohydrodynamic equations for nonlinear bipolar fluids}

\author[P. A. Razafimandimby]{Paul Andr\'e RAZAFIMANDIMBY}

\email[P.A. Razafimandimby]{paul.razafimandimby@unileoben.ac.at}

\address{Department of Mathematics and Information Technology\\
Montanuniversit\"at Leoben\\ Franz Josef Strasse  18, 8700 Leoben,
     Austria}

\date{\today}

\begin{document}
\begin{abstract}
In this article we study the long-time behaviour of a system of nonlinear Partial Differential Equations (PDEs) modelling
the motion of incompressible, isothermal and conducting bipolar fluids in presence of magnetic field. We mainly prove the
existence of a global attractor denoted by $\A$ for the semigroup associated to the aforementioned systems of nonlinear PDEs. We also show that this semigroup
is uniformly differentiable on $\A$. This fact enables us to go further and prove that the attractor $\A$ is of finite-dimensional and we give an explicit bounds for its
Hausdorff and fractal dimensions.
\end{abstract}
 \maketitle
\section{Introduction}
Magnetohydrodynamics (MHD) is a branch of continuum mechanics
which studies the motion of conducting fluids in the presence of
magnetic fields. The system of PDEs in
MHD is basically obtained through the coupling of the dynamical
equations of the fluids with the Maxwell's equations which are used
to take into account the effect of the Lorentz force due to the
magnetic field (see for example \cite{BISKAMP, Chandrasekhar}). The MHD equations play a
fundamental role in Astrophysics, Geophysics, Plasma Physics, and
in many other areas in applied sciences. In many of these, the MHD
flow exhibits a turbulent behavior which is amongst the very
challenging problems in nonlinear science. Due to the folklore fact that the Navier-Stokes (NS) system is
 an accurate model for the motion of incompressible and turbulence in many practical situation, most of scientists have assumed that the fluids in these MHD equations
follow the Newtonian law in which the reduced stress tensor
${\mathbf{T}}(\mathcal{E}(\bu))$ is a linear function of the strain rate $\mathcal{E}(\bu)=\frac{1}{2}\left(\mathbf{\ L}+\mathbf{L}^{\text{T}%
}\right),\quad \mathbf{L}=\nabla{\bu}$, where $\bu$ is the velocity of the fluids. In this way we obtain the conventional system for MHD which has been the object of intensive mathematical research
since the pioneering work of Ladyzhenskaya and Solonnikov
\cite{LADY-SOLO}. We only cite \cite{DESJARDINS}, \cite{LEBRIS}, \cite{SERMANGE}, \cite{STUPELIS} for few
relevant examples of articles about conventional MHD; the reader can consult \cite{GERBEAU+LEBRIS}
for a recent and detailed review.

 There are two major problems that arise when studying the conventional MHD equations or hydrodynamic equations. First, it is well known that the uniqueness of weak solution of
the three-dimensional NSEand MHD equations is still an open problem. As it is always not possible to prove the existence of global attractor in
the case of lack uniqueness of solution, this is an obstacle for the investigation of the long-time behavior which is very important for the understanding
 of some physical features (such as the turbulence in
hydrodynamics) of the models. We refer, for instance, to
\cite{BABIN+VISHIK}, \cite{Robinson}, and \cite{Temam-Infinite} for
some results in this direction for the case of autonomous
Navier-Stokes and other autonomous equations of mathematical
physics. Second,
there are a lot of conducting materials
 that cannot be characterized by Newtonian fluids. To overcome these two problems one generally has to use other model of fluids or some perturbation of the Newtonian law. This has motivated
scientists to consider (conducting) fluid models that allow $%
{\mathbf{T}}$ being a nonlinear function  of $\mathcal{E}(\bu)$. Fluids in
the latter class are called Non-Newtonian fluids. In \cite{LADY1} and \cite%
{LADY2}, Ladyzhenskaya considered a model of nonlinear fluids whose reduced
stress tensor ${\mathbf{T}}(\mathcal{E}(\bu))$ satisfies
\begin{equation*}
{\mathbf{T}}(\mathcal{E}(\bu))=2(\varepsilon+\mu_0|\nabla \bu|^r)\frac{%
\partial \bu_i}{\partial x_j}, \,\, r>0.
\end{equation*}
Since then this model has been the object of intense mathematical analysis
which have generated several important results. We refer to \cite{DU}, \cite{LIONS} for some relevant examples. In \cite{DU} the authors emphasized
important reasons for considering such model. Despite its mathematical success the Ladyzhenskaya model has received a lot of negative criticisms from physicists. Indeed
the Ladyzhenskaya model is a mathematical model used to overcome the lack of uniqueness for the NS equations; it does not have really
a physical meaning  as it does not satisfy some basic principles of continumm mechanics (the frame indifference principle) and thermodynamics (the Clausius-Duhem inequality).
Ne\v cas, Novotn\'y and Silhavy
\cite{NECAS1}, Bellout, Bloom and Ne\v cas \cite{BELLOUT2} has
developed the theory of multipolar viscous fluids which was based
on the earlier work of Ne\v cas and \v Silhav\'y \cite{NECAS2}. Their theory is
compatible with the basic principles of thermodynamics such as the
Clausius-Duhem inequality and the principle of frame indifference,
and their results up to date indicate that  the theory of
multipolar fluids may lead to a better understanding of
hydrodynamic
turbulence (see for example \cite{BELLOUT3}). Bipolar fluids whose reduced stress tensor ${%
\mathbf{T}}(\mathcal{E}(\bu))$
is defined by
\begin{equation}\label{tens-bip}
{%
\mathbf{T}}(\mathcal{E}(\bu))=2\mu_0(\varepsilon+ |\mathcal{E}(\bu)|^2)^{-\frac{\alpha}{2}}%
\mathcal{E}(\bu)-2\mu_1\Delta\mathcal{E}(\bu),
\end{equation}
form a particular class of multipolar fluids.  If $0\le \alpha<1$ then
the fluids are said shear thinning, and shear thickening when
$\alpha<0$.

 In this article we
are aiming to the long-time behaviour of the system of nonlinear PDEs representing the motion of a conducting nonlinear bipolar fluids in presence of magnetic fields.
 For this purpose, we consider a simply-connected and bounded domain $\Omega$ of $\mathbb{R}^n$ ($n=2,3$) such that the boundary $\partial
\Omega$ is of class $C^\infty$. This will ensure the existence of a
normal vector $\bn$ at each of $\partial \Omega$. We are interested in the analytic study of the following system
\begin{equation}  \label{I}
\begin{cases}
\frac{\partial\bu}{\partial t}-\Div \mathbf{T}+\bu\cdot \nabla \bu+\mu \bb\times\curl \bb+ \nabla P=f \text{ in }\Omega\times (0,\infty),  \\
\frac{\partial \bb}{\partial t}+\left(S \curl \curl \bb+\mu \bu\cdot \nabla \bb-\mu \bb\cdot \nabla \bu\right)=0 \text{ in } \Omega\times (0, \infty)\\
\Div \bu=\Div \bb=0\text{   in } \Omega\times [0,\infty),\\
\bu=\tau_{ijl}\bn_j\bn_l=0  \text{ on } \partial \Omega\times [0,\infty),\\
\bb\cdot \bn=\curl \bb \times \bn =0 \text{ on } \partial \Omega \times [0,\infty),\\
\bu(0)=\bu_0, \,\,\, \bb(0)=\bb_0 \text{ in } \Omega,
\end{cases}%
\end{equation}
where $\bu=(u_i; i=1,\ldots, n)$, $\bb=(b_i;i=1,\dots, n)$ and $P$ are functions defined on $\Omega \times [0,T]$,
representing, respectively, the fluid velocity, the magnetic field
and the modified pressure, at each point of $\Omega\times [0,\infty]$. $S$
and $\mu$ are positive constants depending on the Reynolds numbers
of the fluid and magnetic fields,  and the Hartman number.
 The quantities $\bu_0$ and $%
\bb_0$ are given initial velocity and magnetic field,
respectively. The
vector $\bn$ represent the normal to $\partial \Omega$ and 
$\tau_{ijl}$ is defined by
\begin{equation*}
\tau_{ijl}=\frac{\partial \mathcal{E}_{ij}(\bu)}{\partial x_l}, i,j, l\in \{1,..,n \}.
\end{equation*}
Finally, $\mathbf{T}$ designates the extra stress
tensor of the Non-Newtonian fluid which is defined by \eqref{tens-bip}. We suppose throughout that $\eps, \mu_0, \mu_1$ are positive constants and $\alpha \in [0,1)$.
Hereafter we set
\begin{equation*}
 \Gamma(\bu)=\mu_0\left(\eps+|\mce(\bu)|^2\right)^{-\frac{\alpha}{2}}.
\end{equation*}
The structure of the nonlinearity of problem \eqref{I} makes it as
interesting as any nonlinear evolution equations of mathematical
physics such as the conventional MHD or the Navier-Stokes equations.

 When $\bb\equiv0$ then
\eqref{I} reduces to the PDEs describing the motion of isothermal
incompressible nonlinear bipolar fluids which has been extensively
investigated during the last two decades. Existence and uniqueness results of weak solution for nonlinear bipolar fluids with homogeneous boundary condition were given in
\cite{BELLOUT1}, \cite{BELLOUT4}. Existence of unique solution and asymptotic stability of the solutions for the case
 for nonhomogeneous boundary condition were investigated in \cite{BELLOUT5}. The long-time behaviour of \eqref{I} with $\bb\equiv 0$ is investigated in \cite{BELLOUT3} for
$\alpha \in [0,1)$ and in \cite{Bloom} for $\alpha \in (-\infty, 0)$.  Both of the authors of \cite{BELLOUT3} and \cite{Bloom} proved the existence of a global attractor;
they also showed that the flow is finite-dimensional by giving explicit bounds for the Hausdorff and fractal dimensions of the global attractor. Existence results
 for multipolar fluids are also given in \cite{MALEK} and \cite{NECAS1}. These are just examples of relevant work related to \eqref{I} with $\bb\equiv 0$, the papers 
 \cite{MALEKetal} and \cite{RUZICKA-2} contain some reviews of results for nonlinear bipolar fluids.

For $\alpha=0,\,\, \mu_1=0,$ \eqref{I} is reduced to MHD equations
which has been the object of intensive mathematical research
since the pioneering work of Ladyzhenskaya and Solonnikov
\cite{LADY-SOLO}. We only cite \cite{STUPELIS},
\cite{SERMANGE}, \cite{LEBRIS}, \cite{DESJARDINS} for few
relevant examples; the reader can consult \cite{GERBEAU+LEBRIS}
for a recent and detailed review.  Assuming that $\mu_1=0$
Samokhin studied the MHD equations arising  from the coupling of
the Ladyzhenskaya model with the Maxwell equations in
\cite{SAMOKHIN}, \cite{SAMOKHIN2}, \cite{SAMOKHIN3}, and
\cite{SAMOKHIN4}. In these papers he  proved the existence of weak
solution of the model for $\alpha\le 1-\frac{2n}{n+2}$. Later on
Gunzburger and his collaborators (see \cite{GUNZBURGER} and \cite{GUNZBURGER2}) generalized the settings of
Samokhin by taking a fluid with a stress tensor having a more
general structure.
The authors of \cite{GUNZBURGER} and \cite{GUNZBURGER2} analyzed the well-posedness and
the control of \eqref{I} still in the case where $\mu_0=0$ and
$\alpha \le 1-\frac{2n}{n+2}$.

To the best of our knowledge the article \cite{PAUL-16} is the only work treating the case $\{\mu_1\neq 0, \bb\not \equiv 0,
\alpha \in [1-\frac{2n}{n+2}, 1)\}$.  In fact $\alpha=2-p$ in \cite{PAUL-16} and the author proved in \cite{PAUL-16} that if $(\bu_0;\bb_0)\in \h$ and $f\in \ve^\ast$, then
 \eqref{I} with $\eps=1$ (but the author's result is still true for any $\eps>0$) has at least a
global weak solution.
This means that there exists a couple $(\bu;\bb)$ such that
\begin{itemize}
 \item $(\bu;\bb)\in L^\infty_{loc}(0,\infty; \ve)\cap L^2_{loc}(0,\infty; \h)$,
\item $(\bu;\bb)$ satisfies
\begin{multline}
 \left(\frac{\partial \bu(t) }{\partial t}, \phi\right)+\mu_1\left(\frac{\partial \mce(\bu(t))}{\partial x_k}, \frac{\partial \mce(\phi)}{\partial x_k}\right)+\left(\gut\mce(\bu(t)),
 \mce(\phi)\right)
+\left(\bu(t)\cdot \nabla \bu(t), \phi\right)\\-\mu \left(\bb(t)\cdot \nabla \bb(t), \phi\right)=(f,\phi),\label{10}
\end{multline}
and
\begin{equation}
\left(\frac{\partial \bb(t) }{\partial t}, \psi\right)-S\left(\Delta \bb(t), \psi \right)+\mu (\bu(t)\cdot \nabla \bb(t), \psi)-\mu (\bb(t)\cdot \nabla \bu(t), \psi)=0,\label{11}
\end{equation}
for any couple $\Phi=(\phi;\psi)$ of smooth solenoidal functions and $t\in [0,T]$, $T>0$. In \eqref{10} and throughout this work summations over repeated indices are enforced.
\end{itemize}
The author also analyzed the long-time behaviour of solutions of \eqref{I} with $\alpha \in [1-\frac{2n}{n+2}, 1)$. He mainly proved the existence of trajectory attractor for the translation semigroup
acting on the trajectories of the set of weak solutions and that
of external forces. 

In the case when $\alpha \in [1-\frac{2n}{n+2}, 1)$ we are not sure of the uniqueness of the weak solution, therefore the author of \cite{PAUL-16} follows
the method in \cite{VISHIK+CHEPYZHOV} (see also, \cite{VISHIK+CHEPYZHOV-2002} and \cite{VISHIK+CHEPYZHOV+WENDLAND}). Howerver, 
for $0\le \alpha<1$ (which correspond to $1<p\le 2$ in the paper \cite{PAUL-16}) we can show that the weak solution is unique. We will formally check this statement in the proof of Theorem 
\ref{DIFF-SEM}. This fact enables us to define the nonlinear semigroup $\{\mathbb{S}(t); t\ge 0\}$ of the solution of \eqref{I}. 
The semigroup is formally defined by `` let $t>0$ to each initial value $(\bu_0; \bb_0)\in \h$ we associate an element $\mathbb{S}(t)(\bu_0;\bb_0)=(\bu(t); \bb(t))$ of $\h$ which is the
 unique weak wolution of \eqref{I}.'' It is our purpose to study the long time behavior of this semigroup. 
Our results are as follow:
\begin{enumerate}
 \item We give the existence of a compact global attractor $\A$ for $\mathbb{S}(t)$. For this purpose we need to find absorbing sets in $\h$ and in $\ve$.
Since $\ve$ is compact in $\h$ then
we can define from classical argument the existence of the compact global attractor $\A$.
\item We show that the semigroup $\ms(t)$ is differentiable with respect to the initial data $(\bu_0; \bb_0)\in \A$. This result is crucial for establishing the next result.
\item We give an estimate for the bounds of the Hausdorff and fractal dimensions of $\A$.
\end{enumerate}
To establish these results we mainly follow the idea in \cite{BELLOUT3} and the classical results for the investigation of long-time behaviour of dissipative PDEs presented in
\cite{BABIN+VISHIK}, \cite{Robinson} and \cite{Temam-Infinite} for example.
We should mention from the very beginning that every calculation we performed is formal, but we can check them rigourously by using the Galerkin approximation and pass to the limit as
it was done in \cite{PAUL-16}. Also, 
even if we drew our inspiration from \cite{BABIN+VISHIK}, \cite{BELLOUT3}, \cite{DUVAUT+LIONS}, \cite{Robinson}, \cite{SERMANGE} and \cite{Temam-Infinite}
the problem we treated here does not fall in
the framework of these main references.
 Besides the usual nonlinear terms of the conventional MHD equations
it contains another nonlinear term which
exhibits the non-linear relationships between the reduced stress
and the rate of strain $\mathcal{E}(\bu)$ of the conducting
fluids. Because of this, the analysis of the behavior of the MHD
model \eqref{I} tends to be much more complicated and subtle than
that of the Newtonian MHD equations and the model in \cite{BELLOUT3}. Hence, we have had to invest
much effort  to prove many important results which do not follow
from the analysis in the aforementioned papers.

The organization of this article is as follows. We introduce the necessary notations for the mathematical theory of
\eqref{I} in the next section. In Section 3 we prove the existence of absorbing sets in $\h$ and in $\ve$ which enables us to show the existence of a global
attractor for the nonlinear semigroup of solutions to \eqref{I}. The uniformly differentiability of $\ms(t)$ on $\A$ is establisehd in Section 4. In the last section we
give explicit bounds for the fractal dimension of the global attractor $\A$.
\section{Preliminary}
We introduce some notations and background following the
mathematical theory
of hydrodynamics (see for instance \cite{Temam}). For any $p\in (1,\infty)$, $\mathbb{L}^p(\Omega)$ and $\mathbb{W}^{m,p}(\Omega%
)$ are the spaces of functions taking values in $\mathbb{R}^n$
such that each component belongs to the Lebesgue space
$L^p(\Omega)$ and the Sobolev spaces $W^{m,p}(\Omega)$,
respectively. For $p=2$ we use $\h^m(\Omega)$ to describe
 $\mathbb{W}^{m,p}(\Omega)$.
The symbols $|\cdot|$ and $(.,.)$ are the $\mathbb{L}^2$-norm and $\mathbb{L}^2$-inner product, respectively. The norm of $\h^{m}(\Omega%
)$ (resp., $\mathbb{W}^{p,m}(\Omega)$) is denoted by $||\cdot||_{m}$ (resp., $\lVert \cdot \rVert_{p,m}$). As usual,
$\mathbb{C}^\infty_0(\Omega)$ is the space of infinitely
differentiable functions having compact support contained in
$\Omega$. The space $\mathbb{W}^{p,m}_0(\Omega)$ is the closure of
$\mathbb{C}^\infty_0(\Omega)$ in $\mathbb{W}^{p,m}(\Omega)$. Now we
introduce the following spaces
\begin{align*}
 \mathcal{V}_1 &=\left\{\bu\in \mathbb{C}^\infty_0(\Omega): \Div \bu=0\right\},\\
\mathbb{H}_1 &=\left\{\bu \in \mathbb{L}^2(\Omega): \Div \bu =0, \bu\cdot \bn= 0 \text{ on } \partial \Omega \right\},\\
\mathbb{V}_1&=\left\{ \bu \in \mathbb{H}^2(\Omega): \Div
\bu=0, \bu=\frac{\partial \bu }{\partial \bn}=0 \text{ on }
\partial \Omega\right\}.
\end{align*}
We also set
\begin{align*}
 \mathcal{V}_2&=\left\{\bb \in \mathbb{C}^\infty_0(\Omega): \Div \bb=0;\, \bb\cdot \bn=0 \text{ on } \partial \Omega\right\},\\
\mathbb{H}_2&=\text{ the closure of $\mathcal{V}_2$ in $\mathbb{L}^2(\Omega)$},\\
\mathbb{V}_{2}&=\left\{ \bb\in \mathbb{H}^1(\Omega): \Div
\bb=0;\, \bb\cdot \bn=0 \text{ on } \partial \Omega\right\},\\
\mathbb{V}_3&=\left\{\bb\in \mathbb{H}^2(\Omega): \Div
\bb=0;\, \bb\cdot \bn=\curl \bb\times \bn=0 \text{ on } \partial \Omega\right\}
\end{align*}
Note that
\begin{equation*}
 \mathbb{H}_1=\mathbb{H}_2.
\end{equation*}
We equip the space $\mathbb{V}_{1}$ with the norm $||\cdot||_{2}$ generated by the usual $\h^2(\Omega)$-scalar product.

On $\mathbb{V}_{2}$ we define the scalar product
\begin{equation*}
 ((\bu, \mathbf{v}))_{\ve_2}=(\curl \bu, \curl \mathbf{v}),
\end{equation*}
which is equivalent to the usual scalar product of $\mathbb{H}^1(\Omega)$. Hence, from now on $\lVert \bb\rVert_1$ we will denote the norm of $\bb\in \ve_2$ in $\h^1(\Omega)$ and
 $\lVert \bb\rVert_{\ve_2}=\lvert \curl \bb\rvert$ as well.

Let
\begin{align*}
 \mathbb{V}=\mathbb{V}_{1}\times \mathbb{V}_2,\\
\mathbb{H}=\mathbb{H}_1\times \mathbb{H}_2.
\end{align*}
The space $\mathbb{H}$ has the structure of a Hilbert space when equipped with the scalar product
\begin{equation}\label{eq1}
 (\Phi,\Psi)=(\bu, \mathbf{v})+(\bb, \mathbf{C}),
\end{equation}
for $\Phi=(\bu; \bb), \Psi=(\mathbf{v}; \mathbf{C})\in \mathbb{H}$.

The space $\mathbb{V}$ is a Banach space with norm
\begin{equation}\label{eq2}
 ||\Phi||_\ve=||\bu||_2+||\bb||_1,
\end{equation}
for $\Phi=(\bu; \bb)\in \ve$.

For any Banach space $X$ we denote by $X^\ast$ its dual space and $\langle \phi, \bu\rangle$ the value of $\phi\in X^\ast$ on $\bu \in X$.

For any any $\bu\in \mathbb{W}^{k+1,p}(\Omega), k\ge0,$ we set
\begin{equation*}
 \mathcal{E}(\bu)=\frac{1}{2}\Big[(\nabla \bu)+(\nabla \bu)^T\Big].
\end{equation*}
Let us recall the following results whose proofs can be found in \cite{RUZICKA}.
\begin{lem}[\textbf{Korn's inequalities}]
 Let $1<p<\infty$, $\beta=(\beta_1, \dots, \beta_l)$ be a multi-index such that $k=\sum_{i=1}^l \beta_i$, and $\Omega\subset \mathbb{R}^n$ be a bouned domain with smooth boundary.
 Then there exists a positive constant $K(\Omega)$ such that
\begin{equation}
 K(\Omega)||\bu||_{p, k+1}\le \left(\int_\Omega \biggl|\frac{\partial^k \mathcal{E}(\bu)}{\partial x_{\beta_l}.\partial x_{\beta_l}}\biggr|^p dx\right)^\frac{1}{p},
\end{equation}
for any $\bu\in \mathbb{W}^{k+1,p}(\Omega)$.
\end{lem}
The following Poincar\'e's inequality is very crucial
\begin{equation}\label{Poincare}
 \lambda_1\lvert \bu\rvert^2\le \lVert \bu\rVert^2_1,
\end{equation}
for any $\bu\in \h^1(\Omega)$. This inequality holds with $\lVert \bu\rVert_2$ for $\bu\in \h^2(\Omega)$.

We also recall that there exist two positive constants $K_1(\Omega), K_2(\Omega)$ such that
\begin{equation}\label{MISC-0}
 K_1(\Omega) \lVert \bb\rVert^2_{2}\le \lvert \Delta \bb\rvert^2\le K_2(\Omega) \lVert \bb\rVert_{2},
\end{equation}
for any $\bb\in \ve_2\cap \h^2(\Omega)$.

There also holds
\begin{equation}\label{MISC}
 |\nabla \phi|^2\le |\phi|^2+ |\curl \phi|^2,
\end{equation}
for any divergence free function $\phi$ satisfying $(\phi\cdot \mathbf{n})_{|_{\partial \Omega}}=0.$
We refer to \cite{LADYZHENSKAYA-1}
 or \cite{LADYZHENSKAYA-2} for the proofs of \eqref{MISC-0} and \eqref{MISC}.
 
 Finally, for a Banach space $\mathbb{Y}$ and a positive number $\rho$ we set $$ B_{\mathbb{Y}}^{\rho}=\{\mathbf{z}\in \mathbb{Y}; |\mathbf{z}|_{\mathbb{Y}}^2 \le \rho^2\}.$$
\section{Existence of the Global attractor $\A$}
In this section we will investigate the existence of a set $\A\subset \h$ which is the global attractor of $\mathbb{S}(t)$ in $\h$.
But first of all we recall some definitions which are taken
from \cite{Temam-Infinite}.
\begin{Def}
 A set $\A\subset \h$ is called an attractor for the semigroup $\mathbb{S}(t)$ if it enjoys the following properties:
\begin{itemize}
 \item $\A$ is invariant; that is, $\mathbb{S}(t)\A=\A,$ for any $t\ge 0$,
\item $\A$ possesses an open neighborhood $\mathcal{U}$ such that for any $\bu_0\in \mathcal{U}$
$$\lim_{t\rightarrow \infty}\inf_{y\in \A} |\mathbb{S}(t)\bu_0-y|=0. $$
\end{itemize}
The set $\A\subset \h$ is said to be a global attractor if it is a compact attractor and attracts all bounded sets of $\h$.
\end{Def}
The notion of an absorbing set is very important for the investigation of the existence of a global attractor $\A$. We also recall the definition of an absorbing set $\mathcal{B}.$
\begin{Def}
 Let $\mathcal{B}$ be a bounded set of $\h$ and $\mathcal{U}$ be an open set containing $\mathcal{B}$. We say that $\mathcal{B}$ is an abosrbing set in $\mathcal{U}$ if for any
bounded set $\mathcal{B}_0\subset \mathcal{U}$, there exists a time $t_0=t_0(\mathcal{B}_0)$ such that $\mathbb{S}(t)\mathcal{B}_0\subset \mathcal{B}$, for any $t\ge t_0$.
\end{Def}
After these definitions we formulate our first main result.
\begin{thm}\label{attractor}
 The semigroup $\mathbb{S}(t):\h\rightarrow \h, t\ge 0,$ has a global attractor $\A$ which is maximal wrt to the inclusion in $\h.$
\end{thm}

To prove this theorem we start by finding absorbing sets in $\h$ and $\ve$.
\begin{lem}
 There exists a ball $B_\h^{\rho_1}\subset \h$ which absorbs all bounded set of $\h.$
\end{lem}
\begin{proof}
The set $\mathcal{V}$ is dense in $\h$, so it is permissible to take $\phi=\bu$ in \eqref{10}. From this operation we obtain that
\begin{multline}\label{12}
  \frac{1}{2}\frac{d}{dt}|\bu(t)|^2+\mu_1\left(\frac{\partial \mce(\bu(t))}{\partial x_k}, \frac{\partial \mce(\bu(t))}{\partial x_k}\right)+(\gut\mce(\bu(t)), \mce(\bu(t)))
\\-\mu (\bb(t)\cdot \nabla \bb(t), \bu(t))-(f,\bu(t))=0,
\end{multline}
where we have used the well-known fact that $(\bw \cdot \nabla \bv, \bv)=0$ for any $\bw\in H^1(\Omega)$ and any divergence free function $\bv\in H^1_0(\Omega)$.
Using a similar reasoning we also have that
\begin{equation}\label{13}
 \frac{1}{2}\frac{d}{dt}|\bb(t)|^2-S\left(\Delta \bb(t), \bb(t)\right)-\mu (\bb(t)\cdot \nabla \bu(t), \bb(t))=0.
\end{equation}
Now summing up \eqref{12} and \eqref{13} side by side yields
\begin{multline}\label{14}
 \frac{1}{2}\frac{d}{dt}\left(|\bu(t)|^2+|\bb(t)|^2\right)+\mu_1 \left(\frac{\partial \mce(\bu(t))}{\partial x_k}, \frac{\partial \mce(\bu(t))}{\partial x_k}\right)+
\left(\gut\mce(\bu(t)), \mce(\bu(t))\right)\\-S (\Delta \bb(t), \bb(t))-(f,\bu(t))=0.
\end{multline}
Here the property $(\bb\cdot \nabla \bb, \bu)=-(\bb\cdot \nabla \bu, \bb)$ was used.
 Since $$ \left(\gu\mce(\bu), \mce(\bu)\right)\ge 0, $$
$$  \mu_1 \left(\frac{\partial \mce(\bu)}{\partial x_k}, \frac{\partial \mce(\bu)}{\partial x_k}\right)\ge \mu_1 K(\Omega)||\bu||^2_2,$$
and $$ -S(\Delta \bb, \bb)\ge S ||\bb||^2_1, $$
it follows from \eqref{14} that
\begin{equation}\label{15}
 \frac{1}{2}\frac{d}{dt}\left(|\bu(t)|^2+|\bb(t)|^2\right)+ \mu_1 K(\Omega) ||\bu(t)||^2_2+ S ||\bb(t)||^2_1\le |f||\bu(t)|.
\end{equation}
 Letting $\nu_0=\min(\mu_1K(\Omega), S)$ we can see from \eqref{15} along with Poincar\'e's inequality \eqref{Poincare} and Young's inequality  that
\begin{equation*}
 \frac{1}{2}\frac{d}{dt}\left(|\bu(t)|^2+|\bb(t)|^2\right)+ \nu_0 \lambda_1 \left(|\bu(t)|^2+ |\bb(t)|^2\right)\le 2 \delta^{-1} |f|^2+ \delta 2^{-1} \left(|\bu(t)|^2+|\bb(t)|^2\right),
\end{equation*}
for any $\delta>0$.
 Choosing $\delta=\nu_0$ we derive from the last estimate that
\begin{equation} \label{16}
 \frac{d}{dt}\left(|\bu(t)|^2+|\bb(t)|^2\right)+ \nu_0 \lambda_1 \left(|\bu(t)|^2+ |\bb(t)|^2\right)\le \frac{4}{\nu_0}|f|^2.
\end{equation}
Setting $\nu_1=\frac{4}{\nu_0}$ and $y(t)=|\bu(t)|^2+|\bb(t)|^2$ with $y(0)=|\bu_0|^2+|\bb_0|^2$, the estimate \eqref{16} is equivalent to
\begin{equation*}
 \frac{d y(t)}{dt}+\nu_0 y(t)\le \nu_1 |f|^2,
\end{equation*}
 which together with Gronwall's lemma implies that
\begin{equation*}
\begin{split}
 y(t)\le y(0) e^{-\nu_1 t} +\frac{\nu_1 |f|^2}{\nu_0} (1-e^{-\nu_1 t})\\
\le y(0) e^{-\nu_1 t} +\frac{\nu_1 |f|^2}{\nu_0}.
\end{split}
\end{equation*}
Since $e^{- \nu_1 t}\rightarrow 0$ as $t\rightarrow \infty$, we can find a time $t_0= t_0(\bu_0, \bb_0)$ such that
\begin{equation}\label{11-xx}
 y(t)\le \frac{2 \nu_1|f|^2}{\nu_0},
\end{equation}
for any $t\ge t_0$.  This means that the ball in $\h$ $$ B_\h^{\rho_1}=\{\mathbf{z}\in \h; |\mathbf{z}|^2 \le \rho_1^2\},$$
 with $\rho_1> \left (\frac{2 \nu_1|f|^2}{\nu_0}\right)^{1/2}$ is an absorbing set in $\h$.
\end{proof}
Having found $B_\h^{\rho_1}$ we prove that there is also a ball $B^{\rho_2}_\ve\subset \ve$ which attracts all bounded sets of $\h$. We postpone the proof
for the next lemma. For now we establish an additional estimate that is going to play an important role in the aforementioned claim.
From \eqref{15} we see that
\begin{equation}\label{17}
 \frac{d y(t)}{dt}+ 2\nu_0 \left(||\bu(t)||^2_2+||\bb(t)||^2_1\right)\le 2|f||\bu(t)|,
\end{equation}
for any $t\ge 0$. Integrating \eqref{17} over $[t, t+r]$, with $r>0$ arbitrary, yields
\begin{equation*}
  y(t+r)+2 \nu_0 \int_t^{t+r} \left(||\bus||^2_2+||\bbs||^2_1\right) ds\le y(t)+2 |f|\int_t^{t+r} |\bus|ds.
\end{equation*}
Whenever $t\ge t_0$, we have that
\begin{equation}\label{18}
 y(t+r)+2 \nu_0\int_t^{t+r} \left(||\bus||^2_2+||\bbs||^2_1\right) ds\le \rho_1^2 +2|f| \rho_1 r.
\end{equation}
From \eqref{18} we easily derive that
\begin{equation}\label{19}
 \int_t^{t+r} \left(||\bus||^2_2+||\bbs||^2_1\right) ds\le \kappa_0(r),
\end{equation}
where $$\kappa_0(r)=\frac{1}{2\nu_0}\left(\rho_1^2+ 2|f|\rho_1 r\right).$$
We will now show that there is also a ball $B_\ve^{\rho_2}\subset \ve$ which absorbs all bounded set of $\h$.
\begin{lem}
 The semigroup $\mathbb{S}(t)$ has a bounded absorbing set $B_\ve^{\rho_2}\subset \ve$.
\end{lem}
\begin{proof}
The proof of the result will consist of several steps.

\noindent \underline{\texttt{Step 1: Uniform estimate wrt $t$ of $\bb$ in $\ve_2$}}\newline
Taking $\psi=-\Delta\bb(t)$ in \eqref{11} we obtain that
\begin{equation*}
 -\left(\frac{\partial \bb(t)}{\partial t}, \Delta \bb(t)\right)+S |\Delta \bb(t)|^2=\mu\left[(\bu(t)\cdot\nabla \bb(t), \Delta \bb(t))-(\bb(t)\cdot \nabla \bu(t), \Delta\bb(t))\right],
\end{equation*}
  which is equivalent to
\begin{equation}\label{20}
 \frac{1}{2}\frac{d }{dt}|\curl \bb(t)|^2+S |\Delta\bb(t)|^2 = \mu\left[(\bu(t)\cdot\nabla \bb(t), \Delta \bb(t))-(\bb(t)\cdot \nabla \bu(t), \Delta\bb(t))\right].
\end{equation}
Since $|\curl \bb(t) |^2$ is equivalent to  $||\bb(t)||^2_1$ for $\bb(t)\in \ve_1$, we can rewrite \eqref{20} in the following form
\begin{equation}\label{21}
 \frac{1}{2}\frac{d}{dt} ||\bb(t)||^2_1 +S |\Delta\bb(t)|^2 \le c_0 \mu\left[(\bu(t)\cdot\nabla \bb(t), \Delta \bb(t))-(\bb(t)\cdot \nabla \bu(t), \Delta\bb(t))\right].
\end{equation}
Now we want to estimate the right hand side of \eqref{21}. To do so we mainly use that fact that if $\bu\in \ve_1$, then $\bu\in L^q(\Omega)$ for any $2\le q<\infty$ and
$\nabla \bu\in L^s(\Omega)$ with $2\le s\le \frac{2n}{n-2}$ ($2\le s<\infty$ if $n=2$.) For the first term we have
\begin{equation}\label{21-a}
\begin{split}
 \mu (u\cdot \nabla \bb, \bb)&\le \mu c_1 |\bu|_{L^\infty}|\nabla \bb| |\Delta \bb|\\
&\le \delta^{-1} (\mu c_1)^2 ||\bu||^2_2 |\nabla \bb|^2+\delta |\Delta\bb|^2.
\end{split}
\end{equation}
 Owing to \eqref{MISC} and \eqref{21-a} we have that
\begin{equation}\label{21-b}
 \mu(\bu\cdot \nabla \bb, \Delta \bb)\le \delta^{-1} (\mu c_2)^2 ||\bu||^2_2 ||\bb||^2_1+\delta |\Delta \bb|^2.
\end{equation}
By H\"older's inequality we see that the second term verifies
\begin{equation}\label{21-c}
\begin{split}
  \mu(\bb\cdot \nabla \bu, \Delta \bb)&\le \mu c_3 |\bb|_{L^4}|\nabla \bu|_{L^4} |\Delta \bb|\\
&\le \delta^{-1} (\mu c_4)^2 ||\bb||^2_1||\bu||^2_2+\delta |\Delta \bb|^2.
\end{split}
\end{equation}
Let us set $\nu_1=2\delta^{-1}\mu \min(c_4,c_2)$. Owing to \eqref{21-b} and \eqref{21-c} we infer from \eqref{21} that
\begin{equation*}
 \frac{1}{2}\frac{d}{dt} ||\bb(t)||^2_1 +S |\Delta\bb(t)|^2\le \nu_2 ||\bb(t)||^2_1||\bu(t)||^2_2+ 2\delta |\Delta \bb(t)|^2,
\end{equation*}
which with an appropriate choice of $\delta $, let us say $\delta=\frac{S}{4}$, implies that
\begin{equation}\label{22}
 \frac{d}{dt}||\bb(t)||^2_1+ S |\Delta \bb(t)|^2\le \nu_1 ||\bb(t)||^2_1 ||\bu(t)||^2_2.
\end{equation}
Because of \eqref{19} we can use the uniform Gronwall lemma (see, for instance, \cite{Temam-Infinite}) in \eqref{22} to get an uniform estimate wrt $t$ of $\bb$ in $\ve_2$
 which does not explode too fast in time. This estimate reads
\begin{equation*}
 ||\bb(t+r)||^2_1\le \frac{\nu_1\kappa_0(r)}{r}e^{\nu_1\kappa_0(r)},
\end{equation*}
for any $r>0$ and $t\ge t_0$. This implies that for any $r>0$ there exists $\kappa_1(r)=\frac{\nu_1\kappa_0(r)}{r}e^{\nu_1\kappa_0(r)}>0$ such that
\begin{equation}\label{23}
 ||\bb(t)||^2_1\le \kappa_1(r),
\end{equation}
for any $t\ge t_0+r$.

\noindent \underline{\texttt{Step 2: Uniform estimate wrt $t$ of $\Delta \bb$ in $L^2(t, t+r; L^2(\Omega))$}}\newline
For $t\ge t_0+r$ we obtain from \eqref{22} and \eqref{23} that
\begin{equation*}
 ||\bb(t+r)||^2_1+S\int_t^{t+r} |\Delta \bbs|^2 ds\le \nu_1 \kappa_1(r) \int_t^{t+r} ||\bus||^2_2 ds + ||\bb(t)||^2_1,
\end{equation*}
which along with \eqref{19} and \eqref{23} enables us to deduce that
\begin{equation*}
 S\int_t^{t+r} |\Delta \bbs|^2 ds\le \kappa_1\left(\nu_1 \kappa_0(r)+1\right).
\end{equation*}
We easily deduce from this last estimate that for any $r>0$
\begin{equation}\label{25}
 \int_t^{t+r} |\Delta \bbs|^2 ds\le \kappa_2(r),
\end{equation}
for any $t\ge t_0+r$.
This inequality will be very useful for the next step.

\noindent \underline{\texttt{Step 3: Uniform estimate wrt $t$ of $\bu$ in $\ve_1$}}\newline
We want to establish an uniform estimate of $\partial \bu/\partial t$ in $\h$. To shorten notation we will set $\bu_t=\partial \bu/\partial t$. We take $\phi=\but$ in \eqref{10} and we get
\begin{equation}\label{26}
\begin{split}
 |\but(t)|^2+\mu_1 \left(\frac{\partial \mce(\bu(t))}{\partial x_k}, \frac{\partial \mce(\but(t))}{\partial x_k}\right)+\left(\gu\mce(\bu(t)), \mce(\but)\right)
+(\bu(t)\cdot \nabla \bu(t), \but(t))\\
=\mu (\bb(t)\cdot \nabla \bb(t), \but(t))+(f,\but(t)).
\end{split}
\end{equation}
Let $\Sigma$ be the potential defined by
$$\Sigma(\mce)=\int_0^{|\mce|^2} \mu_0 (\eps+s)^{-\alpha/2} ds.$$
it is not difficult to see that
$$ \frac12 \frac{d}{dt}\Sigma (\mce(\bu(t)))=\mu_0 (\eps + |\mce(\bu(t))|^2)^{-\alpha/2}\mce(\bu(t))\mce(\but(t)),$$
 and $$ \left(\frac{\partial \mce(\bu(t))}{\partial x_k}, \frac{\partial \mce(\but(t))}{\partial x_k}\right)
=\frac12 \frac{d}{dt}\left(\frac{\partial \mce(\bu(t))}{\partial x_k}, \frac{\partial \mce(\bu(t))}{\partial x_k}\right)$$
Hence we deduce from \eqref{26} that
\begin{equation}
 \begin{split}
  |\but(t)|^2+\frac{d}{dt}\left(\frac12 \int_\Omega \Sigma(\mce(\bu(t) )dx+\frac{\mu_1}{2} \left(\frac{\partial \mce(\bu(t))}{\partial x_k}, \frac{\partial \mce(\bu(t))}{\partial x_k}\right) \right)\\
\le -(\bu(t)\cdot \nabla \bu(t), \but(t))+\mu (\bb(t)\cdot \nabla \bb(t), \but(t))+ |f||\but(t)|,
 \end{split}\label{27}
\end{equation}
which immediately implies that
\begin{equation}\label{28}
 \begin{split}
  \frac{1}{2}|\but(t)|^2+\frac{d}{dt}\left(\frac12\int_\Omega \Sigma(\mce(\bu(t) )dx+\frac{\mu_1}{2} \left(\frac{\partial \mce(\bu(t))}{\partial x_k}, \frac{\partial \mce(\but(t))}{\partial x_k}\right) dx\right)\\
\le -(\bu(t)\cdot \nabla \bu(t), \but(t))+\mu (\bb(t)\cdot \nabla \bb(t), \but(t))+ \frac{1}{2}|f|^2,\\
\le I_1+I_2+ \frac{1}{2}|f|^2.
 \end{split}
\end{equation}
For $I_1$ we have
\begin{equation}
\begin{split}
 |I_1|&\le c_5 |\bu(t)|_{L^\infty}|\nabla \bu(t)||\but(t)|,\\
&\le c_6 \delta^{-1}||\bu(t)||^2_2|\nabla \bu(t)|^2 + \delta |\but(t)|,\label{28-a}
\end{split}
\end{equation}
for any $\delta>0$.
We can also check that
\begin{equation*}
 \begin{split}
  |I_2|\le \delta^{-1} c_7 |\bb(t)|^2_{L^\infty}|\nabla \bb(t)|^2+\delta |\but(t)|^2,
 \end{split}
\end{equation*}
for any $\delta>0$. Thanks to \eqref{MISC-0},  $||\bb||^2_2$ is equivalent to $|\Delta \bb|^2$ for any $\bb\in \ve_2\cap \ve_3$. Therefore, we can derive from the last estimate and \eqref{MISC} that 
\begin{equation}\label{28-b}
 |I_2|\le \delta^{-1} c_8 |\Delta \bb(t)|^2||\bb(t)||_1^2+\delta |\but(t)|^2.
\end{equation}
By choosing $\delta=1/8$, we can deduce from \eqref{28}-\eqref{28-b} that
\begin{equation}\label{28-c}
 \begin{split}
  \frac{1}{4}|\but(t)|^2 + \frac{dz(t)}{dt}\le 8c_6 ||\bu(t)||_2^2 |\nabla \bu|^2+ 8c_8 |\Delta \bb(t)|^2 ||\bb(t)||^2_1,
 \end{split}
\end{equation}
where we have set $$z(t)= \frac12 \int_\Omega \Sigma(\mce(\bu(t) )dx+\frac{\mu_1}{2} \left(\frac{\partial \mce(\bu(t))}{\partial x_k}, {\partial \mce(\bu(t))}{\partial x_k}\right).$$
Noticing that $$\frac{\mu_1}{2}K(\Omega) \lVert \bu(t)\rVert^2_2 \le \int_\Omega \Sigma(\mce(\bu(t) )dx+\frac{\mu_1}{2} \left(\frac{\partial \mce(\bu(t))}{\partial x_k}, \frac{\partial \mce(\bu(t))}{\partial x_k}\right),$$
 and dropping out the term $1/4 |\but(t)|^2$ we deduce from \eqref{28-c} that
\begin{equation}\label{29}
 \frac{d z(t)}{dt}\le \frac{8 c_6}{2 \mu_1}z(t)+8c_8 |\Delta \bb(t)|^2 ||\bb(t)||^2_1.
\end{equation}
 From here we want to use the Uniform Gronwall lemma, so we need to check that for certain $t^\prime _0$ and  any $r>0$ there exist positive constants $a_1$, $a_2$, $a_3$ such that
\begin{align}
 \frac{8c_6}{\mu_1}\int_t^{t+r} |\nabla \bus|^2 ds\le a_1,\label{29-a}\\
8c_8 \int_t^{t+r} |\Delta \bbs|^2 ds||\bbs||^2_1 ds\le a_2,\label{29-b}\\
\int_t ^{t+r}z(s) ds\le a_3,\label{29-c}
\end{align}
for any $t\ge t^\prime_0$.
Thanks to previous estimate we take $t_0^\prime=t_0+r$ and infer from \eqref{19} that
\begin{equation*}
\begin{split}
 \frac{8c_6}{\mu_1}\int_t^{t+r} |\nabla \bus|^2 ds\le \frac{8c_9}{\mu_1}\int_t^{t+r} ||\bus||^2_2 ds\\
\le \frac{8c_9}{\mu_1} \kappa_0(r).
\end{split}
\end{equation*}
Setting $a_1=\frac{8c_9}{2 \mu_1}\kappa_0(r)$, we get \eqref{29-a}.
Invoking \eqref{23} and \eqref{25} we see that
\begin{equation*}
\begin{split}
 8c_8 \int_t^{t+r} |\Delta \bbs|^2 ||\bbs||_1^2 ds\le 8c_8 \sup_{s\ge t_0+r} ||\bbs||^2_1\int_t^{t+r} |\Delta \bbs|^2 ds, \\
\le 8c_8 \kappa_1(r)\kappa_2(r).
\end{split}
\end{equation*}
Letting $a_2=8c_8 \kappa_1(r)\kappa_2(r)$, we obtain \eqref{29-b}.
Now let us deal with \eqref{29-c}. Thanks to \eqref{14} and \eqref{11-xx} we have that
\begin{equation}\label{30}
\begin{split}
 y(t+r)+2\mu_1 \int_t^{t+r} \left(\frac{\partial \mce(\bus) }{\partial x_k}, \frac{\partial \mce(\bus)}{\partial x_k} \right) ds+2
\int_t^{t+r}\left( \int_\Omega \gu |\mce(\bus)|^2\right)ds\\+2S \int_t^{t+r} ||\bus||^2_1 ds
\le 2 \rho_1 |f|+y(t),
\end{split}
\end{equation}
for any $t\ge t_0+r$. Keeping only the second term of the left hand side of \eqref{30} and using \eqref{11-xx}  we see that
\begin{equation}\label{31}
 \mu_1 \int_t^{t+r} \left(\frac{\partial \mce(\bus) }{\partial x_k}, \frac{\partial \mce(\bus)}{\partial x_k} \right) ds\le {\rho_1}(|f|+\rho_1).
\end{equation}
Now we need to estimate the term involving $\Sigma(\mce(\bu))$. For this purpose we notice that the function $g(s)=\mu_0(\eps+s)^{-\alpha/2}$ is decreasing for $s\in [0,\infty)$. Hence
\begin{equation*}
 \Sigma(\mce(\bu)) \le \int_0^{|\mce(\bu)|^2} (\sup_{s\ge 0} g(s))ds=\frac{\mu_0}{\eps^{\alpha/2} } |\mce(\bu)|^2.
\end{equation*}
 This implies that for any $t\ge t_0+2$
\begin{equation*}
 \int_t^{t+r}\left(\frac12 \int_\Omega \Sigma(\mce(\bus) ) \right)ds\le \frac{\mu_0}{2\eps^{\alpha/2}}\int_t^{t+r} |\mce(\bus)|^2 ds.
\end{equation*}
Invoking Korn's inequality and \eqref{19} we have that
\begin{equation*}
 \int_t^{t+r}\left(\frac12 \int_\Omega \Sigma(\mce(\bus) )\right)ds\le \frac{\mu_0}{2\eps^{\alpha/2}} K(\Omega)\kappa_0(r)r.
\end{equation*}
So putting $a_3={\rho_1}(|f|+\rho_1)+\frac{\mu_0}{2\eps^{\alpha/2}}K(\Omega)\kappa_0(r)r$, we get \eqref{29-c}. Now we can apply the Uniform Gronwall lemma to \eqref{29} and we get
\begin{equation}\label{32}
 z(t+r)\le \left(\frac{a_3}{r}+a_2\right) \exp(a_1),
\end{equation}
for any $t\ge t_0+r$. Korn's inequality along with \eqref{32} implies that for any
\begin{equation}\label{33}
 ||\bu(t)||^2_2\le \frac{1}{\mu_1 K(\Omega)}\left(\frac{a_3}{r}+a_2\right)\exp(a_1),
\end{equation}
for any $t\ge t_0+2r$.

Let us set $$\kappa_3(r)=\kappa_1(r)+\frac{1}{\mu_1 K(\Omega)}\left(\frac{a_3}{r}+a_2\right)\exp(a_1).$$
 So combining \eqref{23} and \eqref{33} we obtain that
\begin{equation}\label{34}
 ||\bu(t)||^2_2+||\bb(t)||^2_1\le \kappa_3(r),
\end{equation}
for any $t\ge t_0+2r$.
The equation \eqref{34} implies that the set
\begin{equation*}
 B_\ve^{\rho_2}=\{\mathbf{z}\in \ve: || \mathbf{z}||^2_{\ve}\le \rho_2^2\},
\end{equation*}
with $\rho_2=\sqrt{\kappa_3(r)}$, is an attracting set in $\ve=\ve_1\times \ve_2$. Here we have set $$ ||\mathbf{z}||^2_{\ve}=||\bu||^2_2+||\bb||^2_1, $$
for any $\mathbf{z}=(\bu;\bb)\in \ve.$
\end{proof}
\begin{proof}[Proof of Theorem \ref{attractor}]
{It is now standard to prove the existence of the global attractor. The set $B_\h^{\rho_1}$ is absorbing in $\h$ and there exists $t_0$ such that
$\mathbb{S}(t)B_\h^{\rho_1}\subset B_\ve^{\rho_2}$, for any $t\ge t_0$. Moreover, $B_\ve^{\rho_2}$ is compact in $\h$, hence $\cup_{t\ge t_0} \mathbb{S}(t) B_\h^{\rho_1}$ is compact and
we can deduce from \cite[Theorem I.1.1]{Temam-Infinite} that $\A=\omega(B_\h^{\rho_1})\quad (=\cap_{t_0\ge 0}\overline{\cup_{t\ge t_0} \mathbb{S}(t) B_\h^{\rho_1}}^{\h})$
is a compact attractor  which is maximal wrt the inclusion in $\h$. This completes the proof of the theorem. }
\end{proof}
Our next concern is to find an estimate of the bounds for $d_H(\A)$ and $d_f(\A)$, the Hausdorff and fractal dimension of $\A$. This procedure needs that $\mathbb{S}(t)$ is Fr\^echet
differentiable wrt to the initial data. Will show this fact in the next section.
\section{Uniform differentiability of $\mathbb{S}(t)$}
Our second result is about the regularity of the semigroup $\{\mathbb{S}(t): t\ge 0\}$ with respect to the initial data of the problem \eqref{I}. We recall the following definition which is
borrowed from \cite{Robinson}.
\begin{Def}\label{Def-Diff}
 We say that $\mathbb{S}(t)$ is uniformly differentiable on $\A$ if for every $(\bu;\bb)\in \A$ there exists a linear operator $\mathfrak{L}(\bu;\bb)$, such that, for all $t\ge 0$;
\begin{equation*}
 \sup_{\{(\bv;\bc),\,\, (\bu;\bb)\in \A:\,\, |\bv-\bu|+|\bc-\bb|< \varkappa \}}\dfrac{\lvert \mathbb{S}(t)(\bv;\bc)-\mathbb{S}(t)(\bu;\bb)-\mathfrak{L}(\bu;\bb)(\bv-\bu;\bc-\bb)\rvert}{\lvert \bv-\bu\rvert
+\lvert \bc-\bb\rvert}\rightarrow 0,
\end{equation*}
  \text{ as } $\varkappa \rightarrow 0$, and
\begin{equation*}
 \sup_{(\bu;\bb)\in\A}\,\, \sup_{(\bv;\bc): |\bv|+|\bc|\neq 0}\dfrac{\lvert \mathfrak{L}(\bu;\bb)(\bv;\bc)\rvert}{\lvert \bv\rvert+\lvert \bc\rvert}<\infty, \text{ for each } t\ge 0.
\end{equation*}
\end{Def}
 The main result of this section is stated in the following
\begin{thm}\label{DIFF-SEM}
 The semigroup $\mathbb{S}(t), t\ge 0$ is uniformly differentiable on $\A$.
\end{thm}
\begin{proof}
We start the proof by showing that $\mathbb{S}(t)$ is Lipschitz continuous. For doing so we consider two weak solutions $(\bu;\bb)$ and $(\bv;\bc)$ associated to the same forcing $f$ and
the initial data $(\bu_0; \bb_0)$ and $(\bv_0;\bc_0)$, respectively. The functions $\bw=\bv-\bu$ and $\bm=\bc-\bb$ satisfy
\begin{multline}
  (\bw_t(t), \phi)+\mu_1\biggl(\frac{\partial \mce(\bw(t))}{\partial x_k}, \frac{\partial \mce(\phi)}{\partial x_k}\biggr)+\biggl(\Gamma(\bv(t))\mce(\bv(t))
-\Gamma(\bu(t))\mce(\bu(t)), \mce(\phi)\biggr)
\\ + (\bv(t)\cdot \nabla \bw(t)+\bw(t)\cdot \nabla \bu(t), \phi)+\mu\biggl[(\bb(t)\cdot\nabla \phi, \bm(t))+(\bm(t)\cdot \nabla \phi, \bc(t)) \biggr]=0,
\end{multline}
and
\begin{multline}
(\bm_t, \psi)-S(\Delta \bm(t), \psi)+\mu\biggl((\bu(t)\cdot \nabla \bm(t)-\bm(t)\cdot \nabla \bu(t)-\bc(t)\cdot\nabla \bw(t)+\bw(t)\cdot\nabla \bc(t), \psi) \biggr)=0,
\end{multline}
where we have set $\bw_t=\dfrac{\partial \bw}{\partial t}$ and $\bm_t=\dfrac{\partial \bm}{\partial t}$.
Taking $\phi=\bw$ and $\psi=\bm$ in the above equations yields
\begin{multline*}
\frac{1}{2}\frac{d}{dt}|\bw(t)|^2+\mu_1\biggl(\frac{\partial \mce\left(\bw(t)\right)}{\partial x_k}, \frac{\partial \mce\left(\bw(t)\right)}{\partial x_k} \biggr)+
\biggl(\Gamma(\bv(t))\mce\left(\bv(t)\right)-\Gamma\left(\bu(t)\right)\mce\left(\bu(t)\right), \mce\left(\bw(t)\right)\biggr)\\-\left(\bw(t)\cdot\nabla \bw(t), \bu(t)\right)
+\mu\biggl[\left(\bb(t)\cdot \nabla \bw(t), \bm(t) \right)+\left(\bm(t)\cdot \nabla \bw(t), \bc(t) \right)\biggr]=0
\end{multline*}
and
\begin{multline*}
 \frac{1}{2}\frac{d}{d t}|\bm(t)|^2+ S ||\bm(t)||^2_1=\mu\biggl(\left(\bm(t)\cdot\nabla \bu(t),
\bm(t)\right)+\left(\bb(t)\cdot \nabla \bw(t), \bm(t)\right)\\+(\bw(t)\cdot \nabla \bm(t), \bc(t))\biggr).
\end{multline*}
Let $\Phi(t)=\lvert\bw(t)\rvert+\lvert\bm(t)\rvert^2$. Summing up side by side and applying Korn's inequality we obtain that
\begin{multline}\label{35}
 \frac{1}{2}\frac{d }{dt}\Phi(t)+\nu_0\left(||\bw(t)||^2_2+||\bm(t)||^2_1\right)\le
 \biggl\lvert \left(\bw(t)\cdot \nabla \bw(t), \bu(t)\right) \biggr\rvert
+\mu\biggl\lvert\left(\bm(t)\cdot \nabla \bw(t), \bc(t)\right)\biggr\rvert\\+\mu\biggl\lvert\left(\bm(t)\cdot \nabla \bu(t), \bm(t)\right)+\left(\bw(t)\cdot \nabla \bm(t), \bc(t)\right)\biggr\rvert
\end{multline}
 To estimate the right hand side of \eqref{35} we will mainly use the fact that if $(\bu_0; \bb_0)$ and $(\bv_0; \bc_0)$ are an element of $\A$ then $\left(\bu(t);\bb(t) \right)$ and
$(\bv(t); \bc(t))$ are element of $B_\ve^{\rho_2}$ for any $t\ge0$. For the first term we have that
\begin{align*}
 \lvert(\bw(t)\cdot \nabla \bw(t), \bu(t) )\rvert\le c_1 \lVert \bu(t)\rVert_2 \lvert \bw(t)\rvert \lvert \nabla \bw(t) \rvert,
\end{align*}
where we have used H\"older's inequality and the embedding $H^2(\Omega)\subset L^\infty(\Omega)$.
By a similar argument we have
\begin{align*}
 \lvert \left(\bm(t)\cdot \nabla \bw(t), \bc(t)\right)+\left(\bw(t)\cdot \nabla \bm(t), \bc(t) \right)\rvert &\le
 c_3 \lvert \bm(t)\rvert \lvert \nabla \bw(t)\rvert_{L^4} \lvert \bc(t)\rvert_{L^4}\\ & \quad \quad +c_1
\lvert \bm(t)\rvert\lvert\nabla \bc(t) \rvert \lvert \bw(t)\rvert_{L^\infty},\\
&\le \tilde{c}_3 \lvert \bm(t)\rvert \lVert \bw(t)\rVert_2 \lVert \bc(t)\rVert_1.
\end{align*}
We can also check that
\begin{align*}
 \lvert\left(\bm(t)\cdot \nabla \bu(t), \bm(t) \right) \rvert&\le c_3 \lvert \bm(t) \rvert \lvert \nabla \bu(t)\rvert_{L^4} \lvert\bm(t)\rvert_{L^4}\\
&\le \bar{c}_3 \lvert \bm(t) \rvert \lVert \bm(t)\rVert_1 \lVert \bu(t)\rVert_2.
\end{align*}
Since $(\bu(t), \bb(t))$ and $(\bv(t), \bc(t))$ are element of $B_\ve^{\rho_2}$, we can find a positive constant $c_{10}$ such that the right hand side of \eqref{35} is bounded from above by
\begin{equation*}
 2\rho_2c_{10}\left(\lvert \bm(t)\rvert \lVert\bw(t)\rVert_2+ \lvert \bm(t) \rvert \lVert\bm(t) \rVert_1  \right).
\end{equation*}
 Hence
\begin{equation*}
 \frac{1}{2}\frac{d }{dt}\Phi(t)+\nu_0\left(||\bw(t)||^2_2+||\bm(t)||^2_1\right)\le 4 \rho_2 c_{10} \biggl(\left( \lvert \bm(t) \rvert +\lvert \bw(t)\rvert \right)
\left(\lVert\bw(t)\rVert_2+\lVert \bm(t)\rVert_1 \right)\biggr),
\end{equation*}
 from which we can infer that
\begin{equation*}
  \frac{d }{dt}\Phi(t)+2\nu_0\left(||\bw(t)||^2_2+||\bm(t)||^2_1\right)\le 2 \varkappa^{-1} (8\rho_2c_{10})^2 \Phi(t)+ \varkappa (\lVert \bw(t) \rVert_2^2+ \lVert \bm(t)\rVert_1^2),
\end{equation*}
for any $\varkappa>0$. By choosing $\varkappa=\nu_0$, we deduce from the last inequality that
\begin{equation*}
  \frac{d }{dt}\Phi(t)+\nu_0\left(||\bw(t)||^2_2+||\bm(t)||^2_1\right)\le 2 \varkappa^{-1} (8\rho_2c_{10})^2 \Phi(t).
\end{equation*}
We put $\eta_1=2 \varkappa^{-1} (8\rho_2c_{10})^2$ and use Gronwall's lemma to deduce that
\begin{equation}\label{36}
 \Phi(t)\le (\lvert\bw_0\rvert^2 +\lvert \bm_0\rvert^2)\exp(\eta_1 t),
\end{equation}
for any $t\ge0$. This show that $\mathbb{S}(t)$ is Lipschitz continuous wrt to the initial data, which also shows the uniqueness of the solution of \eqref{I}.

Let $(\bu; \bb)$ a weak solutions of \eqref{I}.  Using the same argument as in \cite{BELLOUT3} it can be shown that the linearization of \eqref{I} about
$(\bu; \bb)$ is given by the following system of linear PDEs:
\begin{multline}\label{II-a}
\left(\frac{\partial \bcu(t)}{\partial t}, \phi \right)+\left(\gut \mce(\bu(t)) -\alpha \mathbf{A}_{ijkl}(\bu(t))\mce(\bu(t)), \mce(\phi)\right)+\mu_1 (\mce(\bu(t)), \mce(\phi))\\
=\mu\biggl(\left(\bcm(t)\cdot \nabla \bb(t), \phi\right)+\left(\bb(t)\cdot \nabla \bcm(t), \phi\right)\biggr)+(\bcu(t)\cdot \nabla \phi, \bu(t))\\+(\bu(t)\cdot \nabla \phi, \bcu(t)),
\end{multline}
\begin{multline}\label{II-b}
\left(\frac{\partial \bcm(t)}{\partial t}, \psi\right)-S(\Delta \bcm(t), \psi)=-\mu\biggl(\left(\bcu(t)\cdot \nabla \bb(t),\psi\right)
+\left(\bu(t)\cdot \nabla \bcm(t), \psi \right) \\ -\left(\bcm(t)\cdot \nabla \bu(t),\psi\right)-\left(\bb(t)\cdot \nabla \bcu(t), \psi\right)\biggr).
\end{multline}
 Here the tensor $\mathbf{A}_{ijkl}$ is defined by
\begin{equation}\label{2.15}
\mathbf{A}_{ijkl}(\bu)=\mu_0\left(\eps+\lvert\mce(\bu)\rvert^2\right)^{1+ \alpha/2}\mce_{ij}(\bu)\mce_{kl}(\bu),
\end{equation}
and it satisfies
\begin{equation}\label{2-16}
\begin{split}
\int_\Omega \biggl(\gu \lvert\mce(\bcu)\rvert^2-\alpha \mathbf{A}_{ijkl}(\bu) \mce_{ij}(\bcu)\mce_{kl}(\bcu)\biggr)dx\\
\ge 2\eps\alpha \mu_0\int_\Omega \frac{\lvert\mce(\bcu)\rvert^2}{\left( \eps+ \lvert \mce(\bu)\rvert^2 \right)^{1+\alpha/2}} dx
+ 2(1-\alpha)\mu_0\int_\Omega \frac{\mce_{ij}(\bcu) \mce_{kl} (\bcu) \mce_{ij} (\bu) \mce_{kl}(\bu)}{\left(\eps +\lvert\mce(\bu)\rvert^2\right)^{\alpha/2}},
\end{split}
\end{equation}
for all $\eps, \mu_0\ge 0,$ $0\le \alpha <1.$  Note that by standard Galerkin method we can show that \eqref{II-a}-\eqref{II-b}
 has a unique solution $(\bcu;\bcm)\in L^\infty_{loc}(0, \infty; \h)\cap L^2_{loc}(0, \infty; \ve)$ which implies the second condition 
 in Definition \ref{Def-Diff}. Now let $\Theta=\bv-\bu-\bcu$ and $\Psi=\bb-\bc-\bcm$. After some algebra
\begin{multline}\label{III-a}
\left(\frac{\partial \Theta(t) }{\partial t}, \phi\right)+\mu_1 \biggl(\frac{\partial \mce(\Theta(t)) }{\partial x_k}, \frac{\partial \mce(\phi)}{\partial x_k}\biggr)
-\biggl(\gut \mce\left(\bcu(t)\right), \mce(\phi) \biggr)\\ +\alpha \left(\mathbf{A}_{ijkl}\left(\bu(t)\right)\mce(\bcu(t)), \mce(\phi)\right)
+\mu\biggl(\left(\bb(t)\cdot \nabla \phi, \Psi(t)  \right)
+\left(\Psi(t)\cdot \nabla \phi, \bb(t)\right)\biggr)
\\+\left(\bv(t)\cdot\nabla \Theta(t)
+\Theta(t)\cdot \nabla \bv(t), \phi \right)-\left(\bw(t)\cdot \nabla \bw(t), \phi\right)+\left(\bm(t)\cdot \nabla \bm(t), \phi \right)=0,
\end{multline}
and
\begin{multline}\label{III-b}
\left(\frac{\partial \Psi(t)}{\partial t}, \psi\right)-S \left(\Delta \Psi(t), \psi\right)+\mu\biggl( \Theta(t)\cdot \nabla \bc(t)-\bb(t)\cdot \nabla \Theta(t)+\bv(t)\cdot \nabla \Psi(t)
\\-\Psi(t)\cdot \nabla \bv(t)+\bm(t)\cdot \nabla \bw(t)-\bcu(t)\cdot \nabla \bm(t), \psi\biggr)=0.
\end{multline}

Taking $\phi=\Theta$ (resp., $\psi=\Phi$) in \eqref{III-a} (resp., \eqref{III-b}) yields
\begin{equation*}
\begin{split}
\frac{1}{2}\frac{d \lvert\Theta(t)\rvert^2}{dt}+\mu_1\biggl(\frac{\partial \mce(\Theta(t) }{\partial x_k}, \frac{\partial \mce\left(\Theta(t)\right)}{\partial x_k}\biggr)+
\biggl(\Gamma(\bv(t))\mce(\bv(t))-\gut\mce(\bu(t)), \Theta(t)\biggr)\\-\biggl(\Gamma(\bu(t))\mce(\bcu(t)), \mce(\Theta(t))\biggr)
+\alpha \left(\mathbf{A}_{ijkl}(\bu(t)) \mce(\bcu(t)), \mce(\Theta(t))\right)+\left(\Theta\cdot \nabla \bv(t), \Theta(t) \right)\\
=\left(\bw(t)\cdot \nabla \bw(t), \Theta(t) \right)
-\mu\biggl(\left(\bb(t)\cdot \nabla \Theta(t), \Psi(t)\right)+\left(\Psi(t)\cdot \nabla \Theta(t), \bb(t)\right)\\+\left(\bm(t)\cdot\nabla \bm(t), \Theta(t)\right)\biggr),
\end{split}
\end{equation*}
and
\begin{multline}
\frac{1}{2}\frac{d \lvert\Psi(t)\rvert^2}{dt}+S \lVert \Psi(t)\rVert^2_1 +\mu\biggl(\left(\Theta(t)\cdot \nabla \bc(t)-\bb(t)\cdot \nabla \Theta(t), \Psi(t) \right)
-\left(\Psi(t)\cdot \nabla \bv(t), \Psi(t) \right)\\+\left(\bm(t)\cdot \nabla \bw(t)-\bcu(t)\cdot \nabla \bm(t), \Psi(t)\right) \biggr)=0
\end{multline}
By using some results from \cite[Page 150-151]{BELLOUT3} we see that
\begin{multline}\label{37}
\frac{1}{2}\frac{d \lvert\Theta(t)\rvert^2}{dt}+\mu_1K(\Omega)\lVert \Theta(t)\rVert^2_2+\left(\Theta(t)\cdot \nabla \bv(t), \Theta(t)\right)-
\left(\bw(t)\cdot \nabla \bw(t), \Theta(t)\right)\\ \le
-\mu\biggl(\left(\bb(t)\cdot \nabla \Theta(t), \Psi(t)\right)+\left(\Psi(t)\cdot \nabla \Theta(t), \bb(t)\right)+\left(\bm(t)\cdot \bm(t), \Theta(t) \right)\biggr)\\+\int_\Omega \Sigma_{ijklmn}\mce_{ij}(\Theta(t))\mce_{kl}(\bw(t))\mce_{mn}(\bw(t)) dx,
\end{multline}
where
\begin{equation}
\Sigma_{ijklmn}=\int_0^1 \int_0^1 \dfrac{\partial^3 \Sigma}{\partial\mce_{ij}\partial \mce_{kl}\mce_{mn}}\left(\mce(\bu(t)+\sigma \tau \bw(t)) \right)\tau d\tau d\sigma
\end{equation}
We also have
\begin{multline}\label{38}
\frac{1}{2}\frac{d \lvert \Psi(t)\rvert^2}{dt}+S\lVert \Psi(t)\rVert^2_1 \le
 \mu\biggl\lvert \left(\Theta(t)\cdot \nabla \bc(t)-\bb(t)\cdot \nabla \Theta(t), \Psi(t)\right)-\left(\Psi(t)\cdot \nabla \bv(t), \Psi(t)\right)
\\+\left(\bm(t)\cdot \nabla \bw(t)-\bcu(t)\cdot \nabla \bm(t), \Psi(t)\right) \biggr\rvert
\end{multline}
Now we want to estimate each term in the right hand side of \eqref{37}.
We have that
\begin{align*}
 \biggl(\Theta(t)\cdot \nabla \bu(t), \Theta(t) \biggr)& \le c_1 \lvert \Theta(t)\rvert \lvert \nabla \Theta(t)\rvert \lvert \bu(t)\rvert_{L^\infty},\\
&\le C(\Omega) \lvert\Theta(t)\rvert \lVert \Theta(t)\rVert_2 \lVert \bu(t)\rVert_2.
\end{align*}
Owing to Young's inequality and the fact that $\bu(t)\in B_\ve^{\rho_2}$ for any $t\ge0$ and , we obtain that
\begin{equation}\label{39}
 \biggl(\Theta(t)\cdot \nabla \bu(t), \Theta(t) \biggr)\le C(\Omega)^2 \rho_2^2 \varkappa^{-1} \lvert\Theta(t)\rvert^2 +\varkappa \lVert \Theta(t)\rVert_2^2,
\end{equation}
for any $\varkappa>0.$

H\"older's inequality implies that
\begin{equation}
 \biggl \lvert \biggl(\bb(t)\cdot \nabla \Theta(t), \Psi(t)\biggr)+\biggl(\Psi(t)\cdot \nabla \Theta(t), \bb(t) \biggr)\biggr\rvert\le 2 C(\Omega)
\lvert\bb(t)\rvert_{L^4}\lvert \nabla \Theta(t)\rvert_{L^4} \lvert \Psi(t)\rvert.
\end{equation}
Since $H^1(\Omega)\subset L^4(\Omega)$ for $n=2,3$, and $\lVert \psi \rVert_{H^1}$ and $\lvert \curl \psi \rvert$ are equivalent for any $\psi\in \ve_2$, 
we deduce from the last inequality that
\begin{equation}\label{40-0}
 \biggl \lvert \biggl(\bb(t)\cdot \nabla \Theta(t), \Psi(t)\biggr)+\biggl(\Psi(t)\cdot \nabla \Theta(t), \bb(t) \biggr)\biggr\rvert\le 2 C(\Omega) \rVert \bb(t)\rVert_1
\lVert \Theta(t)\rVert_2 \lvert \Psi(t)\rvert.
\end{equation}
 Since $(\bu;\bb)\in B_\ve^{\rho_2}$, we easily see that $\lVert \bb(t)\rVert_1\le \rho_2$. Therefore
\begin{equation}\label{40}
 \mu\biggl \lvert \biggl(\bb(t)\cdot \nabla \Theta(t), \Psi(t)\biggr)+\biggl(\Psi(t)\cdot \nabla \Theta(t), \bb(t) \biggr)\biggr\rvert
\le \varkappa^{-1} \left(2 \mu C(\Omega)\rho_2 \right)^2\lvert\Psi(t)\rvert^2+\varkappa \lVert \Theta(t)\rVert_2^2.
\end{equation}
 It is not difficult to show that
\begin{equation*}
 \mu \biggl\lvert\biggl(\bb(t)\cdot \nabla \bm(t), \Theta(t)\biggr)\biggr\lvert\le \mu C(\Omega)\lVert \bm(t)\rVert_1 \lVert \Theta(t)\rVert_2 \lvert \bm(t)\rvert.
\end{equation*}
 Since $(\bw; \bm)\in B_\ve^{\rho_2}$, we have $\lVert \bm\rVert_1 \le 2 \rho_2$. Thus,
\begin{equation}\label{41}
 \mu \biggl\lvert\biggl(\bb(t)\cdot \nabla \bm(t), \Theta(t)\biggr)\biggr\lvert\le \varkappa^{-1} \left(2\mu \rho_2 C(\Omega)\right)^2\lvert \bm(t)\rvert^2+ \varkappa \lVert \Theta(t)\rVert_2^2.
\end{equation}
We also check that
\begin{align}
 \biggl\lvert \biggl(\bw(t)\cdot \nabla \bw(t), \Theta(t)\biggr)\biggr\rvert&\le C(\Omega) \lvert \bw(t)\rvert \lVert \Theta(t)\rVert_2\rVert \bw(t)\rVert_2,\nonumber\\
&\le \varkappa^{-1} \left(2C(\Omega)\rho_2\right)^2\lvert \bw(t)\rvert^2 +\varkappa \lVert \Theta(t)\rVert_2^2. \label{42}
\end{align}
As far as the term involving $\Sigma_{ijklmn}$ is concerned we have that
\begin{equation}\label{43}
 \biggl\lvert \int_\Omega \Sigma_{ijklmn} \mce_{ij}\left(\Theta(t) \right)\mce_{kl}\left(\bw(t)\right) \mce_{mn}\left(\bw(t)\right)dx \biggr\rvert
\le \varkappa^{-1} \left(2C(\Omega) \rho_2\right)^2\lVert \bw(t)\rVert_2^2 +\varkappa \lVert \Theta(t)\rVert^2_2,
\end{equation}
where we have used again the  fact that  $\lVert \bw(t)\rVert_2\le 2\rho_2$. Putting \eqref{39}-\eqref{43} in \eqref{37} yields
\begin{multline}\label{44}
 \frac{1}{2}\frac{d \lvert \Theta(t) \rvert^2}{dt}+ \mu_1 K(\Omega) \lVert \Theta(t)\rVert_2^2 \le \varkappa^{-1} \left(C(\Omega) \rho_2\right)^2
 \biggl(\lvert \Theta(t)\rvert^2+
 4\mu^2 \lvert\Psi(t)\rvert^2+4\mu^2 \lvert \bm(t)\rvert^2\\+4 \lvert\bw(t)\rvert^2+4\lVert \bw(t)\rVert_2^2
\biggr)+5\varkappa \lVert \Theta(t)\rVert^2_2
\end{multline}
In the next few lines we estimate each term in the right hand side of \eqref{38}.
Owing to the same argument we used for \eqref{40-0} and the fact that $\lvert \nabla \bb\rvert\le \lvert \curl \bb \rvert+\lvert \bb\rvert$ we obtain that
\begin{align*}
 \biggl\lvert \biggl(\Theta(t)\cdot \nabla \bc(t)-\bb(t)\cdot \nabla \Theta(t), \Psi(t) \biggr)\biggr \rvert &\le C(\Omega) \biggl(\lvert \nabla \bc(t)\rvert
\lVert \Theta(t)\rVert_2\lvert \Psi(t)\rvert
\\ &\quad \quad + \lVert \bb(t)\rVert_1 \lVert \Theta(t)\rVert_2 \lvert \Psi(t)\rvert\biggr),\\
& \le C(\Omega)\biggl( \left(\lVert \bc(t)\rVert_1+ \lvert\bc(t) \rvert \right)\lVert \Theta(t)\rVert_2\lvert \Psi(t)\rvert
 \\ & \quad \quad + \lVert \bb(t)\rVert_1 \lVert \Theta(t)\rVert_2 \lvert \Psi(t)\rvert \biggr).
\end{align*}
 Since $(\bu;\bb)$ and $(\bv;\bc)$ belong to $B_\ve^{\rho_2}$ we have that
\begin{align}
 \biggl\lvert \biggl(\Theta(t)\cdot \nabla \bc(t)-\bb(t)\cdot \nabla \Theta(t), \Psi(t) \biggr)\biggr \rvert & \le 3 \mu \rho_2 C(\Omega) \lvert\Psi(t)\rvert \lVert \Theta(t)\rVert_2\nonumber \\
&\le \varkappa^{-1} \left(3 \mu \rho_2 C(\Omega)\right)^2 \lvert\Psi(t)\rvert^2+\varkappa \lVert \Theta(t)\rVert^2_2.\label{45}
\end{align}
Using the similar idea as used for \eqref{40-0} and \eqref{45} we see that
\begin{align}
 \mu \biggl\lvert \left(\Psi(t)\cdot \nabla \bu(t), \Psi(t) \right) \biggr\rvert&\le\mu C(\Omega) \lvert \Psi(t)\rvert_{L^4} \lvert\nabla \bu(t)\rvert_{L^4}\lvert \Psi(t)\rvert,\nonumber\\
&\le \mu C(\Omega) \lVert \Psi(t)\rVert_1 \lVert \bu(t)\rVert_2 \lvert \Psi(t)\rvert,\nonumber\\
&\le \gamma^{-1} \left(\mu\rho_2 C(\Omega)\right)^2 \lvert \Psi(t)\rvert^2+\gamma \lVert \Psi(t)\rVert^2_1.\label{46}
\end{align}
Before we proceed further we recall that $\bcu=\bw-\Theta$. Hence
\begin{equation*}
 \biggl(\bm(t)\cdot \nabla \bw(t)-\bcu(t) \cdot \nabla \bm(t), \Psi(t) \biggr)=\biggl(\bm(t)\cdot \nabla \bw(t)-\bw(t)\cdot \nabla \bm(t)+\Theta(t) \cdot \nabla \bm(t), \Psi(t)\biggr).
\end{equation*}
As in \eqref{46} we can check that
\begin{equation}\label{47}
 \mu \biggl\lvert \biggl(\bm(t)\cdot \nabla \bw(t), \Psi(t) \biggr)\biggr\rvert\le \gamma^{-1} \left(2\mu\rho_2 C(\Omega)\right)^2\lvert \bm(t)\rvert^2+\gamma \lVert \Psi(t)\rVert_1^2.
\end{equation}
We also have that
\begin{align*}
 \biggl\lvert \biggl(\bw(t)\cdot \nabla \bm(t), \Psi(t)\biggr)\biggr\rvert& \le C(\Omega) \lvert \bw(t)\rvert_{L^\infty} \lvert \nabla \Psi(t)\rvert \lvert \bm(t)\rvert,\\
&\le C(\Omega) \biggl(\lVert \bw(t) \rVert_2 \lVert \Psi(t)\rVert_1 \lvert\bm(t)\rvert+ \lVert\bw(t)\rVert_2\lvert \bm(t)\rvert\lvert \Psi(t)\rvert \biggr).
\end{align*}
 If $(\bu_0;\bb_0)\in \A$ and $(\bv_0;\bc_0)\in \A$, then $(\bw(t);\bm(t))\in B_\ve^{2\rho_2}$. Hence we derive from the last estimate that
\begin{equation}
 \mu \biggl\lvert \biggl(\bw(t)\cdot \nabla \bm(t), \Psi(t)\biggr)\biggr\rvert \le \gamma^{-1} \left(2\mu\rho_2 C(\Omega) \right)^2 \lvert \bm(t)\rvert^2
+\gamma \lVert \Psi(t)\rVert_1^2. \label{48}
\end{equation}
Finally,
\begin{equation*}
 \biggl\lvert \biggl(\Theta(t)\cdot \nabla \bm(t), \Psi(t)\biggr)\biggr\rvert\le C(\Omega) \lVert \Theta(t)\rVert_2\lvert\Psi(t)\rvert \left(\lVert \bm(t)\rVert_1+\lvert \bm(t)\rvert\right).
\end{equation*}
Since $(\bw(t);\bm(t))\in B_\ve^{2\rho_2}$, we infer from the last estimate that
\begin{align}
 \mu \biggl\lvert \biggl(\Theta(t)\cdot \nabla \bm(t), \Psi(t)\biggr)\biggr\rvert & \le 2\mu \rho_2 C(\Omega) \lVert \Theta(t)\rVert_2 \lvert \Psi(t)\rvert,\nonumber\\
&\le \varkappa^{-1} \left(2\mu \rho_2 C(\Omega) \right)\lvert \Psi(t)\rvert^2+\varkappa \lVert \Theta(t)\rVert_2^2. \label{49}
\end{align}
Using \eqref{45}-\eqref{49} in \eqref{38} we find that
\begin{multline}\label{50}
 \frac{1}{2}\frac{d \lvert \Psi(t)\rvert^2}{dt}+S \lVert \Psi(t)\rVert_1^2\le 2 \varkappa \lVert \Theta(t)\rVert^2_2+3\gamma \lVert \Psi(t)\rVert_1^2
+ \left(\mu\rho_2 C(\Omega) \right)^2\left(13\varkappa^{-1} + \gamma^{-1}\right)\lvert \Psi(t)\rvert^2\\+ 2\gamma^{-1} \left(2\mu\rho_2 C(\Omega) \right)^2\lvert \bm(t)\rvert^2.
\end{multline}
Letting $Y(t)=\lvert \Theta(t)\rvert^2+\lvert \Psi(t)\rvert^2$ and summing up \eqref{44} and \eqref{50} side by side implies that
\begin{equation}\label{51}
\begin{split}
 \frac{1}{2}\frac{d Y(t)}{ dt}+\mu_1 \lVert \Theta(t)\rVert_2^2+S\lVert \Psi(t)\rVert_1^2\le 
 \varkappa^{-1} \left(\rho_2C(\Omega)\right)^2 \lVert \bw(t)\rVert^2_2
+7 \varkappa \lVert \Theta(t)\rVert^2_2 + 3\gamma\lVert\Psi(t)\rVert_1^2\\ +
\left(\mu \rho_2 C(\Omega) \right)^2\left(17\varkappa^{-1}+\gamma^{-1} \right)\lvert \Psi(t)\rvert^2+ 
\varkappa^{-1} \left(C(\Omega)\right)^2 \lvert \Theta(t)\rvert^2
\\+ \varkappa^{-1} \left(2\rho_2 C(\Omega) \right)^2 \lvert \bw(t)\rvert^2+ \left(2\mu \rho_2 C(\Omega) \right)^2\left(\varkappa^{-1}+2\gamma^{-1} \right)\lvert \bm(t)\rvert^2.
\end{split}
 \end{equation}
Choosing $\varkappa=\dfrac{\mu_1}{14}$ and $\gamma=\dfrac{S}{6}$. We deduce from \eqref{51} that there exist positive constants $\tilde{C}_i, i=1,\dots, 5,$ depending only on
$\Omega, \mu, S, \mu_1, \rho_2$ such that
\begin{equation}
 \begin{split}
 \frac{1}{2}\frac{d Y(t)}{ dt}+\frac{\mu_1}{2} \lVert \Theta(t)\rVert_2^2+\frac{S}{2} \lVert \Psi(t)\rVert_1^2\le \tilde{C}_1 \lvert \Theta(t)\rvert^2
+ \tilde{C}_2\lvert \Psi(t)\rvert^2 +\tilde{C}_3 \lvert \bw(t)\rvert^2 + \tilde{C}_4 \lVert \bw(t)\rVert^2_2\\ +\tilde{C}_5 \lvert \bm(t)\rvert^2. \label{52}
 \end{split}
\end{equation}
 Let us set $ \chi_1=2 \max(\tilde{C}_1, \tilde{C}_2),$ $ \chi_2=2\max(\tilde{C}_3, \tilde{C}_4),$ and $\chi_3=2 \max(\tilde{C}_5,1)$. We infer from \eqref{52} that
\begin{equation*}
 \frac{d Y(t)}{dt} +\mu_1 \lVert \Theta\rVert^2_2+S \lVert \Psi(t)\rVert^2_1\le \chi_1 Y(t)+ \chi_2 \lvert Z(t)\rvert^2+\chi_3 \lVert Z(t)\rVert^2_{\ve},
\end{equation*}
where
\begin{align*}
 \lvert Z(t)\rvert^2=\lvert\bw(t)\rvert^2 +\lvert \bm(t)\rvert^2,\\
\lVert Z(t)\rVert^2_\ve=\lVert \bw(t)\rVert^2_2+\lVert \bm(t)\rVert_1^2,\\
|Z_0|^2=|\bw_0|^2+|\bm_0|^2.
\end{align*}
Dropping out the second term in the left hand side of \eqref{52} and invoking Gronwall's lemma yield
\begin{equation*}
 Y(t)\le \chi_4 e^{\chi_1 t}\int_0^t \left(\lvert Z(s)\rvert^2+\lVert Z(s)\rVert^2_{\ve}  \right)ds, \forall t\ge0,
\end{equation*}
where $\chi_4=\max(\chi_2, \chi_3).$  From \eqref{36} we have
\begin{equation*}
 \lvert Z(t)\rvert^2\le |Z_0|^2 e^{\eta_1 t},
\end{equation*}
 and
\begin{align*}
 \nu_0\int_0^t \lVert Z(s)\rVert^2_{\ve}ds & \le \eta_1 \int_0^t \lvert Z(s) \rvert^2 ds,\\
&\le |Z_0|^2 \left(e^{\eta_1 t}-1\right),\\
&\le |Z_0|^2 e^{\eta_1 t}.
\end{align*}
Therefore
\begin{align*}
 Y(t)\le  \chi_4 e^{\left( \eta_1+ \chi_1 \right)t} |Z_0|^2\left(\frac{1}{\eta_1}+\frac{1}{\nu_0}\right).
\end{align*}
And this estimate implies that for any
\begin{equation*}
\dfrac{\lvert \bw(t)-\bcu(t) \rvert^2 + \lvert \bm(t)-\bcm(t)\rvert^2 }{\lvert \bu_0-\bv_0\rvert+\lvert\bb_0-\bc_0\rvert }\rightarrow 0,
\end{equation*}
for any $t\ge 0$ as $\lvert \bu_0-\bv_0\rvert\rightarrow 0$ and $\lvert \bb_0-\bc_0\rvert\rightarrow 0$. This means that the semigroup $\mathbb{S}(t):\h\rightarrow \h$ is uniformly
 differentiable wrt to the initial data.
\end{proof}
\section{Bounds for the fractal $d_f(\A)$ and Hausdorff $d_H(\A)$ dimensions of $\A$}
 The fractal dimension of $\A$ is written as $d_f(\A)$. It is basically based on the number of closed balls of a fixed radius $\delta$ needed to cover $\A$ (see, for instance,
\cite[Section 13.1.1]{Robinson}). We denote by $N(\A, \delta)$ the minimum number of balls in such a cover. We recall the following definition for sake of precision.
\begin{Def}[See, for e.g., \cite{Robinson}]
 If the closure of $\A$ is compact, then the fractal dimension of $\A$ is given by
\begin{equation*}
 d_f(\A)={\lim\sup}_{\delta\rightarrow 0}\dfrac{\log N(\A,\delta)}{\log(1/\delta)},
\end{equation*}
where we allow the limit in the above equation to take the value $+\infty$.
\end{Def}

Now we recall the definition of the Hausdorff dimension of $\A$. It is based on the approximation of the $d-$dimensional volume of $\A$ by a covering of finite number of balls
$B(x_i,r_i)$ with radii $r_i\le \delta.$ Let $$ \mu(\A, d, \delta)=\inf\biggl\{\sum_{i}r_i^d: r_i\le \delta \text{ and } \A\subset \cup_{i}B(x_i,r_i) \biggr\}.$$
\begin{Def}[See, for e.g., \cite{Robinson}]
 The Hausdorff dimension of $\A$ is given by
\begin{equation*}
 d_H(\A)=\inf_{d>0}\{d:\lim_{\delta\rightarrow 0}\mu(\A, d, \delta)=0 \}.
\end{equation*}
\end{Def}
We have the following relation between fractal and Hausdorff dimensions.
\begin{thm}[See, for e.g., \cite{Robinson}]
$d_H(\A)\le d_f(\A).$
\end{thm}

To calculate the bounds for the fractal dimension $d_f(\A)$ of $\A$ we will mainly follow the scheme in \cite[Section 13.2]{Robinson}. For this purpose we let $\biggl\{\Phi_i=\begin{pmatrix}
          \phi_i\\ \psi_i
         \end{pmatrix}: i=1,2,\dots\biggr\}$ be an orthonormal basis of $\h$ and $P_m$ be an orthogonal projection from $\h$ onto $$\Span \biggl\{\Phi_i: i=1,2,\dots,m\biggr\}.$$
For a solution $(\bu;\bb)$ of \eqref{I} we denote by $\mathcal{L}(\bu;\bb)$ the mapping defined by
\begin{equation*}
\frac{d}{dt}\begin{pmatrix}
\bcu(t)\\
\bcm(t)
\end{pmatrix}=\fd\cdot \begin{pmatrix}
\bcu(t)\\
\bcm(t).
\end{pmatrix}
\end{equation*}
Following the argument in \cite[chapter V]{Temam-Infinite} we denote by $\{\Lambda_i; i \in \mathbb{N}\}$ the set of Lyapunov exponents of $\h$ and we set
$$ \tilde{q}_m=\lim_{t\rightarrow \infty}\sup_{(\bu_0;\bb_0)\in \A}\frac{1}{t}\int_0^t -\tr(\fd\circ P_m)(s) ds.$$
It is a standard fact (see, for instance, \cite[Chapter V]{Temam-Infinite}) that
\begin{equation*}
\Lambda_1+\dots +\Lambda_m\le -\tilde{q}_m.
\end{equation*}
 Therefore, it follows from \cite[Theorem 13.16]{Robinson} (See also, \cite[Subsection V.3.4]{Temam-Infinite}) that
\begin{thm}
 If $-\tilde{q}_m< 0$, then $m$ is the smallest positive integer such that $d_f(\A)\le m.$
\end{thm}
The value of $\fd$ at $(\phi_i, \psi_i)^T$ is defined by $\fd(\phi_i, \psi_i)^T=(X;Y)$, where
\begin{multline*}
 (X,\phi)=-\mu_1\biggl(\frac{\partial \mce(\phi_i)}{\partial x_k},\frac{\partial \mce(\phi)}{\partial x_k} \biggl)
-\biggl(\gu\mce(\phi_i)-\alpha \mathbf{A}_{ijkl}\left(\bu(t)\right)\mce(\phi_i), \mce(\phi)\biggr)\\+ \biggl(\bu(t)\cdot \nabla \hp,\phi \biggr)+\biggl(\hp\cdot \nabla \bu(t), \phi\biggr)
+\mu\biggl(\Sp\cdot \nabla \phi
+\bb(t)\cdot \nabla \phi, \Sp \biggr),
\end{multline*}
and
\begin{multline*}
(Y,\psi)=S(\Delta \Sp, \psi)- \mu\biggl( \biggl(\hp\cdot \nabla \bb(t), \psi\biggr)+\biggl(\bu(t)\cdot \nabla \Sp, \psi\biggr)
-\biggl(\Sp\cdot \nabla \bu(t), \psi\biggr)\\-\biggl(\bb(t)\cdot \nabla \hp, \psi\biggr) \biggr).
\end{multline*}
for any $(\phi;\psi)\in \ve$.
Therefore $$ \biggl(\fd(\hp; \Sp)^T, (\hp; \Sp)\biggr)=(X,\hp)+(Y, \Sp),$$ which is equivalent to
\begin{equation*}
\begin{split}
 \biggl(\fd(\hp; \Sp)^T, (\hp; \Sp)\biggr)=-\mu_1 \biggl(\frac{\partial \mce(\hp)}{\partial x_k}, \frac{\partial \mce(\hp)}{\partial x_k}\biggr)
 -\biggl(\gu \mce(\hp)-\alpha \mathbf{A}_{ijkl}\mce(\hp), \mce(\hp)\biggr)\\ -S \lvert \nabla \Sp\rvert^2  +\biggl(\hp\cdot\nabla \bu(t), \hp\biggr)- \mu \biggl( \biggl(\Sp\cdot \nabla \hp, \bb(t)\biggr)+
\biggl(\hp\cdot \nabla \bb(t), \Sp \biggr)\\-\biggl(\Sp\cdot \nabla \bu(t), \Sp\biggr) \biggr).
\end{split}
\end{equation*}
Let us recall \cite[Equation (2.16)]{BELLOUT3}
\begin{multline*}
 I_0=\int_\Omega \biggl(\gut \lvert \mce(\hp)\rvert^2-\alpha \mathbf{A}_{ijkl}\left(\bu(t)\right) \lvert \mce(\hp)\rvert^2\biggr)dx
\ge 2\eps \mu_0 \int_\Omega \frac{\lvert \mce(\hp)\rvert^2}{\left(\eps+ \lvert \mce\left(\bu(t)\right)\rvert^2 \right)^{1+\alpha/2}}dx
\\ + 2(1-\alpha)\mu_0\int_\Omega \frac{\lvert \mce(\hp)\rvert^2}{\left(\eps + \lvert \mce\left(\bu(t)\right)\rvert^2\right)^{\alpha/2}} dx
\end{multline*}
 Let us set $I_1$ (resp., $I_2$) the first (resp., second) term in the above inequality.
Both of $I_1$ and $I_2$ are positive, so we have
\begin{multline*}
 -\biggl(\fd(\hp; \Sp)^T, (\hp; \Sp)\biggr)\ge \mu_1 \biggl(\frac{\partial \mce(\hp) }{\partial x_k}, \frac{\partial \mce(\hp) }{\partial x_k}\biggr)+S \lvert \nabla \Sp\rvert^2+2(1-\alpha) I_2\\
-\biggl(\hp\cdot \nabla \bu(t), \hp \biggr)-\mu\biggl(\biggl(\Sp\cdot \nabla \hp, \bb(t) \biggr)+ \biggl( \hp\cdot \nabla \bb(t), \Sp\biggr)\\-\biggl(\Sp\cdot \nabla \bu(t), \Sp
\biggr)\biggr).
\end{multline*}
Since $0\le \alpha <1$, we easily see that $p=2-\alpha\in (1,2]$. We can check that
\begin{equation*}
 \left(2(1-\alpha)\right)^{-1} I_2\ge \biggl(\int_\Omega \left(\eps +\lvert \mce\left(\bu(t)\right)\rvert^2\right)^{p/2}dx\biggr)^{\frac{p-2}{p}}\biggl(\int_\Omega \lvert\mce(\hp)\rvert^p dx \biggr)^{2/p}.
\end{equation*}
from which along with Korn's inequality we derive that
\begin{equation*}
 \left(2(1-\alpha)\right)^{-1} I_2\ge  \biggl(\int_\Omega \left(\eps +\lvert \mce\left(\bu(t)\right)\rvert^2\right)^{p/2}dx\biggr)^{\frac{p-2}{p}} K(\Omega) \lVert \hp\rVert^2_{W^{1,p}}.
\end{equation*}
For $0\le \alpha <1$ we have that
\begin{equation}
 \int_\Omega \left(\eps + \lvert \mce\left(\bu(t)\right) \rvert^2 \right)^{p/2}dx\le C(\Omega,\alpha, \eps)+ \int_\Omega \lvert \mce\left(\bu(t)\right)\rvert^p dx.
\end{equation}
As $1< p\le 2$ we also see that
\begin{equation}
 \int_\Omega \left(\eps + \lvert \mce\left(\bu(t)\right) \rvert^2 \right)^{p/2}dx \le C(\Omega, \alpha, \eps)+ C(\Omega)\biggl(\int_\Omega \lvert \mce\left(\bu(t)\right) \rvert^2 dx \biggr)^{p/2}
\end{equation}
Since $(\bu; \bb)$ is a solution of \eqref{I}, $\bu$ is an element of $L^2(0,\infty; \ve_1)$. Therefore
\begin{equation*}
 \int_\Omega \left(\eps + \lvert \mce\left(\bu(t)\right)\rvert^2 \right)^{p/2}dx\le C(\Omega, \alpha, \eps)+ C(\Omega,\alpha),
\end{equation*}
   for almost everywhere $t\ge 0$. Thanks to the fact that $p-2\ge 0$, we see that
\begin{equation*}
 \left(2(1-\alpha)\right)^{-1} I_2\ge \left(\tilde{C}(\Omega,\alpha,\eps)\right)^{\frac{p-2}{p}}K(\Omega) \lVert \hp\rVert^2_{W^{1,p}},
\end{equation*}
 where $\tilde{C}(\Omega,\alpha,\eps)=C(\Omega, \alpha, \eps)+C(\Omega, \alpha)$. Hence
\begin{equation*}
\begin{split}
 -\tr(\fd\circ P_m)=-\sum_{i=1}^m \biggl( \fd \Phi_i, \Phi_i \biggr)\ge \mu_1 \sum_{i=1}^m \biggl(\frac{\partial \mce(\hp) }{\partial x_k},\frac{\partial \mce(\hp) }{\partial x_k}\biggr)
+S \sum_{i=1}^m \lvert \nabla \Sp\rvert^2\\+ \tilde{K}(\Omega)\sum_{i=1}^m \lVert \hp\rVert^2_{W^{1,p}}- \sum_{i=1}^m F(\bu,\bb,\hp,\Sp),
\end{split}
\end{equation*}
where $\tilde{K}(\Omega)=2(1-\alpha) \biggl(\tilde{C}(\Omega, \alpha,\eps) \biggr)^{\frac{p-2}{p}}$ and
$$ F(\bu,\bb, \hp, \Sp)=(\hp\cdot \nabla \bu(t), \hp)+\mu\biggl( \left(\Sp\cdot \nabla \hp, \bb(t) \right)+ \left( \hp\cdot \nabla \bb(t), \Sp \right)+ \left(\Sp\cdot \nabla \Sp, \bu(t)
\right) \biggr).$$
From H\"older's inequality and Young's inequality we derive that
\begin{align*}
 \left(\hp\cdot \nabla  \bu(t) , \hp\right)&\le C(\Omega)\lvert \bu(t)\rvert_{L^\infty} |\hp| ||\hp||_{H^1},\\
&\le \varkappa^{-1}_1 C(\Omega)^2 ||\bu(t)||^2_{2}+ \varkappa_1 |\hp|^2||\hp||^2_{H^1},
\end{align*}
for any $\varkappa_1>0.$
Since $|\hp|^2+|\Sp|^2=1$ we obtain from the last estimate that
\begin{equation*}
 \left(\hp\cdot \nabla  \bu(t) , \hp\right)\le \varkappa_1^{-1} C(\Omega)^2 ||\bu(t)||^2_2+ \varkappa_1 ||\hp||^2_{H^1}.
\end{equation*}
By a similar argument we have that
\begin{align*}
 \mu\biggl(\left( \Sp\cdot \nabla \hp, \bb(t) \right)+ \left(\hp\cdot \nabla \bb(t),\Sp \right) \biggr)& \le 2 \mu C(\Omega) |\Sp| ||\Sp||_{H^1}||\bb(t)||_{1}\\
& \le \gamma^{-1} \left(2\mu C(\Omega)\right)^2 ||\bb(t)||^2_1+ \gamma || \hp||^2_{H^2},
\end{align*}
for any $\gamma>0$.
Also
\begin{align*}
 \mu \left(\Sp\cdot \nabla \Sp, \bu(t) \right)&\le \mu C(\Omega) ||\Sp||_{1} ||\bu(t)||_2 |\Sp|,\\
&\le \varkappa_2^{-1} \left(\mu C(\Omega) \right)^2 ||\bu(t)||^2_{2}+\varkappa_2 ||\Sp||^2_{H^1},
\end{align*}
for any $\varkappa_2>0.$
Thus
\begin{multline*}
 F(\bu,\bb,\hp,\Sp)\le \varkappa_1 ||\hp||^2_{H^1}+\gamma ||\hp||^2_{H^2}+\varkappa_2||\Sp||^2_{H^1}+ C(\Omega)^2 \left(\varkappa_1^{-1}+\varkappa_2^{-1} \mu^2 \right)||\bu(t)||^2_2\\+
\gamma^{-1} \left(2 \mu C(\Omega)\right)^2 ||\bb(t)||^2_1.
\end{multline*}
This last inequality and Korn's inequality  enable us to state that
\begin{equation*}
 \begin{split}
 -\tr(\fd\circ P_m)\ge \mu_1 K(\Omega) \sum_{i=1}^m ||\hp||^2_{H^2}+S\sum_{i=1}^m ||\Sp||^2_{H^1}+ \tilde{K}(\Omega)\sum_{i=1}^m ||\hp||^2_{W^{1,p}}\\
-\gamma \sum_{i=1}^m ||\hp||^2_{H^2}-\varkappa_1 \sum_{i=1}^m ||\hp||^2_{H^1}-\gamma^{-1} \left(\mu C(\Omega)\right)^2 ||\bb(t)||^2_2\\
-\varkappa_2\sum_{i=1}^m ||\Sp||^2_{H^1}-C(\Omega)^2 \left(\varkappa_1^{-1}
+\varkappa_2^{-1} \mu^2 \right)||\bu(t)||^2_2.
 \end{split}
\end{equation*}
It was proved in \cite{BELLOUT3} that there exists a positive constant $d(\Omega)$ such that for any $\sigma>0$
\begin{equation*}
 |\phi|^2_{W^{1,p} }\ge  \begin{cases} \frac{1}{d(\Omega)}||\phi||^2_{H^1} \text{ if } \alpha=0,\\
\frac{\sigma^{1/\delta^\prime}}{\delta^\prime d(\Omega)}||\phi||^2_{H^1}-\sigma^{\frac{1-\delta^\prime}{\delta^\prime}}\left(\frac{1-\delta^\prime}{\delta^\prime}\right)||\phi||^2_{H^2} \text{ otherwise},
 \end{cases}
\end{equation*}
with $\delta^\prime=2\left(\frac{2-\alpha}{4+\alpha}\right)$.
Therefore
\begin{multline*}
 -\tr(\fd \circ P_m)\le \biggl(\mu_1 K(\Omega)-\tilde{K}(\Omega) \sigma^\frac{1-\delta^\prime}{\delta^\prime} \frac{1-\delta^\prime}{\delta^\prime}\biggr)\sum_{i=1}^m ||\hp||^2_{H^2}
-\gamma \sum_{i=1}^m ||\hp||^2_{H^2}\\+\left(\tilde{K}(\Omega) \frac{\sigma^\frac{1}{\delta^\prime} }{\delta^\prime d(\Omega)}-\varkappa_1\right)\sum_{i=1}^m ||\hp||^2_{H^1}
+ (S-\varkappa_2)\sum_{i=1}^m ||\Sp||^2_{H^1}\\- C(\Omega)^2 \left(\varkappa_1^{-1}+ \varkappa_2^{-1} \mu^2\right)||\bu(t)||^2_2-\gamma^{-1} (4 \mu C(\Omega))^2 ||\bb(t)||^2_1.
\end{multline*}
As $\sigma$ and $\gamma$ are arbitrary, we can choose them in such a way that
\begin{equation*}
 \mu_1 K(\Omega)-\tilde{K}(\Omega) \sigma^\frac{1-\delta^\prime}{\delta^\prime} \frac{1-\delta^\prime}{\delta^\prime}=\frac{\mu_1 K(\Omega)}{2},
\end{equation*}
that is $\sigma=\biggl(\dfrac{\mu_1 K(\Omega) \delta^\prime }{\tilde{K}(\Omega) (1-\delta^\prime) } \biggr)^\frac{\delta^\prime}{1-\delta^\prime}$ and $\gamma=\dfrac{\mu_1 K(\Omega) }{2}$.
 This choice implies that
\begin{multline*}
 -\tr(\fd\circ P_m)\ge
\biggl[\dfrac{\tilde{K}(\Omega) }{\delta^\prime d(\Omega)}\biggl(\dfrac{\mu_1 K(\Omega) \delta^\prime}{\tilde{K}(\Omega)(1-\delta^\prime) } \biggr)^\frac{1}{1-\delta^\prime}
-\varkappa_1 \biggr]\sum_{i=1}^m ||\hp||^2_{H^1}+ (S-\varkappa_2)\sum_{i=1}^m ||\Sp||^2_{H^1}\\- C(\Omega)^2 \left( \varkappa_1^{-1} +\varkappa_2^{-1} \mu^2 \right)||\bu(t)||^2_2-
\frac{8 \mu^2}{\mu_1 K(\Omega)}||\bb(t)||^2_1.
\end{multline*}
Choosing $\varkappa_2=S$ and $\varkappa_1=\frac{\gamma^\prime}{2}$, where
\begin{equation}\label{def-gamp}
\gamma^\prime=\dfrac{\tilde{K}(\Omega) }{\delta^\prime d(\Omega)}\biggl(\dfrac{\mu_1 K(\Omega) \delta^\prime}{\tilde{K}(\Omega)(1-\delta^\prime) } \biggr)^\frac{1}{1-\delta^\prime},
\end{equation}
we infer from the last estimate that
\begin{equation*}
 -\tr(\fd\circ P_m)\ge \frac{\gamma^\prime}{2}\sum_{i=1}^m ||\hp||^2_{H^1}-C(\Omega)^2 \left(S^{-1} + \frac{2 \mu^2}{\gamma^\prime} \right)||\bu(t)||^2_2
-\frac{8\mu^2}{\mu_1 K(\Omega)}||\bb(t)||^2_1.
\end{equation*}
Choosing $\hp$ as the eigenfunctions of the Stokes operator whose eigenvalues satify
\begin{equation}
 \lambda_j\ge \tilde{c} \lambda_1 j^n,
\end{equation}
we derive from \cite{Robinson} that
\begin{equation}
 \sum_{i=1}^m ||\hp||^2_{H^1}\ge \tilde{c} \lambda_1 m^{1+2/n}.
\end{equation}
Thus
\begin{equation}
 -\tr(\fd\circ P_m)\ge \frac{\gamma^\prime \tilde{c} \lambda_1}{2} m^{1+2/n}-C(\Omega)^2\left( S^{-1}+ \frac{2\mu^2}{\gamma^\prime}\right)||\bu(t)||^2_2
-\frac{8\mu^2}{\mu_1 K(\Omega)} ||\bb(t)||^2_2.
\end{equation}
To give the bounds for the fractal dimension of $\A$ we also need to estimate the time average of $\lVert\bu(t)\rVert_2^2$ and $\lVert \bb(t)\rVert_1^2$.
From \eqref{14} we derive that
\begin{equation*}
\begin{split}
 \lvert \bu(t)\rvert+ \lvert \bb(t)\rvert+2\mu_1K(\Omega)\int_0^t \lVert \bu(s)\rVert_2^2 ds+ 2 S\int_0^t \lVert \bb(s)\rVert_1^2 ds\le 4\gamma^{-1} \lvert f\rvert^2
+ \gamma \int_0^t \lvert \bu(s)\rvert^2 ds\\+ \lvert\bu_0\rvert^2+\lvert \bb_0\rvert^2.
\end{split}
\end{equation*}
 By using Poincar\'e's inequality and choosing $\gamma=\dfrac{\mu_1K(\Omega)}{2\lambda_1}$we can deduce from the above inequality that
\begin{align*}
\lim_{t\rightarrow \infty} \frac{1}{t}\int_0^t \lVert \bu(s)\rVert_2^2 ds\le \frac{\Lambda}{\mu_1 K(\Omega)},\\
\lim_{t\rightarrow \infty} \frac{1}{t}\int_0^t \lVert \bb(s)\rVert_1^2 ds\le \frac{\Lambda}{S},
\end{align*}
where \begin{equation}\label{def-lambda}
\Lambda= \frac{4 \lambda_1 \lvert f\rvert^2}{\mu_1 K(\Omega)}.
      \end{equation}
Now we see that
\begin{equation}
 -\tilde{q}_m\le \biggl(\frac{C(\Omega)^2}{\mu_1 K(\Omega)} \left(\frac{1}{S}+\frac{2\mu^2}{\gamma^\prime}\right)+ \frac{8\mu^2}{\mu_1 K(\Omega)S}\biggr)\Lambda-
\frac{\gamma^\prime \tilde{c} \lambda_1}{2} m^{1+2/n}.
\end{equation}
Owing to \cite[Theorem 13.16]{Robinson} we see that $d_f(\A)\le m$ if $m$ is the smallest positive integer such that  $-\tilde{q}_m<0$, that is
\begin{equation}
 m-1 <\biggl[\frac{2\Lambda}{\gamma^\prime \tilde{c} \lambda_1 \mu_1 K(\Omega)}\biggl( C^2(\Omega)
\left(\frac{1}{S}+\frac{2 \mu^2}{\gamma^\prime} \right)+\frac{8\mu^2}{S}\biggr)\biggr]^{ \frac{n}{n+2} } < m,
\end{equation}
 where $\gamma^\prime$ and $\Lambda$ are given respectively by \eqref{def-gamp} and \eqref{def-lambda}, and $\delta^\prime=2\left(\dfrac{2-\alpha}{4+\alpha}\right)$.

The conclusion of this section is stated in the following claim that we have just proved above.
\begin{thm}\label{BOUNDS}
 The global attractor $\A$ for the MHD equations for the nonlinear bipolar fluids \eqref{I} is finite-dimensional with
$d_H(\A)\le d_f(\A)\le m$ where $m>0$ is positive integer given by
\begin{equation}
 m-1 <\biggl[\frac{2\Lambda}{\gamma^\prime \tilde{c} \lambda_1 \mu_1 K(\Omega)}\biggl( C^2(\Omega)
\left(\frac{1}{S}+\frac{2 \mu^2}{\gamma^\prime} \right)+\frac{8\mu^2}{S}\biggr)\biggr]^{ \frac{n}{n+2} } < m,
\end{equation}
 and $\gamma^\prime$ and $\Lambda$ are given respectively by \eqref{def-gamp} and \eqref{def-lambda}, and $\delta^\prime=2\left(\dfrac{2-\alpha}{4+\alpha}\right)$.
\end{thm}

\section*{Acknowledgment}
The author's research is supported by the Austrian Science Foundation through the Lise-Meitner-Programm  M1487.

\end{document}